\documentclass[11pt]{article}
\usepackage[centertags]{amsmath}
\usepackage{amsfonts}
\usepackage{amssymb,amsmath,amsfonts}
\usepackage{booktabs}
\usepackage{array, color}
\usepackage{placeins}
\usepackage{float}
\usepackage{relsize}
\usepackage{lipsum}
\usepackage{graphicx}
\usepackage{subcaption}
\usepackage{float}
\restylefloat{table}
\usepackage{derivative}
\usepackage{multirow}
\usepackage{amssymb,amsmath,amsfonts,fancyhdr}
\usepackage{amsmath}
\usepackage{bigints}
\numberwithin{equation}{section}
\usepackage{amsthm}
\usepackage{mathtools}

\usepackage[utf8]{inputenc}
\usepackage[english]{babel}
\newtheorem{theorem}{Theorem}[section]
\newtheorem{corollary}{Corollary}[theorem]
\newtheorem{remark}{Remark}[theorem]
\newtheorem{lemma}{Lemma}[section]
\usepackage{bm}
\usepackage{amssymb}

\usepackage{booktabs}
\usepackage{color}
\usepackage{subfloat}
\usepackage{hyperref}
\usepackage{multirow}

\usepackage{lscape}
\usepackage{natbib}
\makeatletter
\newcommand{\thickhline}{%
    \noalign {\ifnum 0=`}\fi \hrule height 1pt
    \futurelet \reserved@a \@xhline
}
\newcolumntype{"}{@{\hskip\tabcolsep\vrule width 1pt\hskip\tabcolsep}}
\makeatother
\makeatletter \oddsidemargin  -.1in \evensidemargin -.1in
\newcommand{\doublespacing}{\let\CS=\@currsize
 \renewcommand{\baselinestretch}{1.05}\tiny\CS}
\begin{document}
\newcommand{\bea}{\begin{eqnarray}}
\newcommand{\eea}{\end{eqnarray}}
\newcommand{\nn}{\nonumber}
\newcommand{\bee}{\begin{eqnarray*}}
\newcommand{\eee}{\end{eqnarray*}}
\newcommand{\lb}{\label}
\newcommand{\nii}{\noindent}
\newcommand{\ii}{\indent}
\newtheorem{thm}{Theorem}[section]
\newtheorem{lem}{Lemma}[section]
\newtheorem{rem}{Remark}[section]
\theoremstyle{remark}
\renewcommand{\theequation}{\thesection.\arabic{equation}}
\vspace{5cm}
\title{\bf Equivariant Estimation of the Selected Guarantee Time}
\author{Masihuddin$^{a}$\thanks{Email : masih.iitk@gmail.com, masihst@iitk.ac.in}
~~	
and~ Neeraj Misra$^{b}$\thanks{Email : neeraj@iitk.ac.in}}
\date{}
\maketitle \noindent {\it $^{a,b}$ Department of Mathematics \& Statistics, Indian
Institute of Technology Kanpur, Kanpur-208016, Uttar Pradesh, India} 

\newcommand{\oddhead}{Equivariant Estimation  of the Guarantee Time}
\renewcommand{\@oddhead}
{\hspace*{-3pt}\raisebox{-3pt}[\headheight][0pt]
{\vbox{\hbox to \textwidth
{\hfill\oddhead}\vskip8pt}}}
\vspace*{0.05in}
\begin{abstract}
 Consider two independent exponential populations having different unknown location parameters and  a common unknown scale parameter. Call the population associated with the larger location parameter as the ``best" population and the population associated with the smaller location parameter as the ``worst" population. For the goal of selecting the best (worst) population a natural selection rule, that has many optimum properties,  is the one which selects the population corresponding to the larger (smaller) minimal sufficient statistic. In this article, we consider the problem of estimating  the location parameter of the population selected using this natural selection rule. For estimating the location parameter of the selected best population, we derive the  uniformly minimum variance unbiased estimator (UMVUE)  and show that the analogue of the best affine equivariant estimators (BAEEs) of location parameters is a generalized Bayes estimator. We provide some admissibility and minimaxity results for estimators in the class of  linear, affine and permutation equivariant estimators, under the criterion of scaled mean squared error. We also derive a sufficient condition for inadmissibility of an arbitrary affine and permutation equivariant estimator.  We provide similar results for the problem of estimating the location parameter of selected population when the selection goal is that of selecting the worst exponential population. Finally, we provide a simulation study to compare, numerically, the performances of some of the proposed estimators. 
\end{abstract}
\noindent {\it AMS 2010 SUBJECT CLASSIFICATIONS:} 62F07 · 62F10 · 62C15 · 62C20\\

\newcommand\blfootnote[1]{%
	\begingroup
	\renewcommand\thefootnote{}\footnote{#1}%
	\addtocounter{footnote}{-1}%
	\endgroup
}\blfootnote{
}
\section{Introduction}
In everyday life, we come across situations where one is interested in choosing the best or the worst 
option/population from the many available choices. After selecting the best or the worst  population among the available populations, using a pre-specified selection rule, a problem of practical interest is of estimation of  some characteristic(s) of the selected population. In the statistical literature these types of problems are often categorized as ``Estimation Following Selection" problems. There have been extensive research studies on ranking and selection and the associated estimation problems in the past seven decades. For a nice introduction to the methodology of the ranking and selection problems, the reader is referred to \cite{bahadur1950problem}, \cite{barr1966introduction} and \cite{gibbons1979introduction}. For an excellent overview of the ranking and selection problems, one may refer to the monographs by \cite{gnp2002multiple} and \cite{gibbon1999selecting}. Over the last four decades, a vast literature can be found on the problems of estimation following selection. For some important references on these problems of estimation, see:  \cite{sarkadi1967estimation}, \cite{165}, \cite{dahiya1974estimation}, \cite{hsieh1981estimating}, \cite{cohen1982estimating}, \cite{sack1984estimation}, \cite{MR898261}, \cite{MRKD}, \cite{vellannals1992inadmissibility},  \cite{MR1319130}, \cite{misra2006estimating}, \cite{stallard2008estimation}, \cite{kumar2009reliability}, \cite{arshad2015estimation} and  \cite{arshad2017estimating}. \par 
 The two parameter negative exponential probability model is used in many reliability and life testing experiments to describe, for example, the failure times of systems or components. For more details, one is referred to \cite{johnson1995continuous} and \cite{balakrishnan2019exponential}. In the context of reliability   engineering problems, one might be interested in selecting a machine having the largest guarantee time and estimating the guarantee time of the selected machine. Consider two different machines producing an item. Suppose that the lifetime of the items produced by the $i^{th}$ machine (constituting the population $\Pi_i$) is described by two parameter exponential distribution having the probability density function (p.d.f.),  
\begin{equation}
f(x|\mu_i, \sigma) = \begin{cases}
\frac{1}{\sigma}e^{-\frac{(x-\mu_i)}{\sigma}}, & \text{if $x\geq \mu_i$ }\\
0, & \text{ otherwise}
\end{cases}~ ,
\end{equation}
where,  the location parameter $\mu_i \in \mathbb{R}$ represents the guarantee time of the items produced by the $i^{th}~(i=1,2)$ machine and $\sigma^{-1} > 0$ represents the common failure rate of  the items produced by the two machines.

Estimation following  selection problems involving exponential populations have particularly attracted the attention of many researchers. Some of the worth mentioning contributions in this direction are due to \cite{misra1993umvue}, \cite{misra1998estimation}, \cite{vellaisamy2003quantile}, \cite{kumar2001estimating} and \cite{arshad2016estimation}. \cite{misra1993umvue}
considered the problem of estimation of location parameter of the selected exponential population from $k~ (\geq2)$ exponential population having different unknown location parameters but a common known scale parameter. \cite{arshad2016estimation} considered the case of known and, possibly, unequal scale parameters and derived some decision theoretic results for the problem of estimating the location parameter of the population selected from $k~ (\geq2)$ populations. \cite{kumar2001estimating} considered the problem of estimating quantiles of a selected exponential population with a common location parameter and different
scale parameters. \cite{vellaisamy2003quantile} addressed the problem of quantile estimation of the selected exponential population with unequal location parameters and a common unknown scale parameter. He studied properties of some natural estimators and derived a sufficient condition for the inadmissibility of a scale equivariant estimator, by applying the method of differential inequalities. The estimators dominating the natural estimator (analogue of the best affine equivariant estimators of $\mu_{1}$, $\mu_{2}$, $\ldots$, $\mu_{k}$) provided in \cite{vellaisamy2003quantile} are not easily expressible in closed form (see Corollary 3.1 in \cite{vellaisamy2003quantile}) and it is also not clear whether one can really obtain dominating estimators that are affine and permutation equivariant. Motivated by this, in this article, we consider the  problem of estimating the location parameter (which is also called  the guarantee time) of the selected population from two exponential populations having different unknown location parameters and a common unknown scale parameter and derive various results in the decision theoretic set up. \par The paper is organized as follows: In Section $2$, we introduce various notations, that are used throughout the paper, and provide formulation of the problem. In Section $3$, we derive the UMVUE and show that the generalized Bayes estimator of the location parameter of the selected population is same as the analogue of BAEE of $\mu_{i}'s$. In Section $4$, we derive some admissibility and minimaxity results under the criterion of scaled mean squared error. A sufficient condition for inadmissibility of an affine and permutation equivariant estimator is also derived. Section $5$ is devoted to the problem of estimating the location parameter of the selected exponential population when the goal is that of selecting the worst
exponential population (population associated with the smaller location parameter). In Section 5, we derive results similar to the ones obtained in Section 4, for the problem of estimation after selection of the best population. In Section $6$, we report a simulation study on the performances of various competing estimators. A concluding discussion is provided  in Section 7 of the paper.
\section{Preliminaries and Formulation}
In order to maintain uniformity in our presentation we will use the following notations throughout the paper:\\
\begin{itemize}
	\item $\mathbb{R}:$ the real line $\left(-\infty,\infty\right)$;
	\item $\mathbb{R}^k:$ the $k$ dimensional Euclidean space, $k \in \{2,3,\ldots\}$;
	\item iid: independent and identically distributed;
	\item For random variables $T_1$ and $T_2$, $T_1 \stackrel{d}{=} T_2$, indicates that $T_1$ and $T_2$ are identically distributed;
	\item $Exp\left(\lambda,  \xi\right)$ : exponential distribution with location parameter $\lambda \in \mathbb{R}$ and scale parameter $\xi \in \left(0,\infty\right) $;
	\item Gamma($\alpha, \tau$) : the gamma distribution with shape parameter $\alpha>0$ and scale parameter $\tau >0$, having pdf
	\begin{equation*}
	g_{\alpha, \tau}(x)=\begin{cases}
	\frac{1}{\Gamma(\alpha) \tau^{\alpha}}e^{-\frac{x}{\tau}}x^{\alpha-1}, & \text{ $x > 0$ }\\
	0, & \text{otherwise}
	\end{cases}
	\end{equation*}
	where $\Gamma(\alpha)$ denotes the usual gamma function.
	\item For real numbers  $a$ and $b$
\begin{equation*}
\label{S1.E1}
I(a\geq b)=\begin{cases}
1, & \text{if $a \geq b$ }\\
0, & \text{if $a < b$}.
\end{cases}
\end{equation*}

\end{itemize}

Let $X_{i1},X_{i2},...,X_{in}$~$(i=1,2)$ be a pair of mutually  independent random samples of  the same size $n~(\geq 2)$,  each, from two exponential populations $\Pi_1$ and $\Pi_2$,  with respective unknown location parameters $\mu_1$ and $\mu_2$ and a common unknown scale parameter $\sigma$, where $\underline{\theta}= \left(\mu_1,\mu_2,\sigma \right) \in \mathbb{R}^2 \times \left(0,\infty\right)=\Theta$, say. Define ${X_i}= \min\limits_{1\leq j \leq n}  X_{i,j}$ , $S_i=\sum_{j=1}^{n}(X_{ij}-X_i)$, ~$i=1,2$ and $S=S_1+S_2$. Here it should be noted that   $\underline{T}=\left(X_1, X_2,S\right)$ is a complete-sufficient~(hence minimal sufficient) statistic for $\underline{\theta} \in \Theta$. Also, $X_1, X_2$ and $S$ are mutually independent, with $ X_i \sim Exp(\mu_i, \frac{\sigma}{n})$, $i=1,2$ and $\frac{S}{\sigma} \sim Gamma(2(n-1),1)$. In addition to the notations introduced above, we make use of the following notations throughout the paper:\par 
  $\underline{X} = (X_1, X_2)$; $Z_1 = \min \{X_1 , X_2\}$ (minimum of $X_1$ and $X_2$);   $Z_2 = \max \{X_1 , X_2\}$ (maximum of $X_1$ and $X_2$); $Z = Z_2 - Z_1$; $W=\frac{Z}{S}$; $V=\frac{S}{\sigma}$; $\theta_1 = \min\{\mu_1, \mu_2\}$;  $\theta_2 = \max\{\mu_1, \mu_2\}$;  $\mu = \frac{n(\theta_2 - \theta_1)}{\sigma}$.  Also, for any $\underline{\theta} \in \Theta$, $\mathbb{P}_{\underline{\theta}}(\cdot)$ will denote the probability  measure induced by $\underline{T} = (X_1, X_2, S)$, when $\underline{\theta} \in \Theta$ is the true parameter value, and $\mathbb{E}_{\underline{\theta}}(\cdot)$ will denote the expectation operator under the probability measure $\mathbb{P}_{\underline{\theta}}(\cdot)$, $\underline{\theta} \in \Theta$. Note  that $\mathbb{P}_{\underline{\theta}}(Z\geq0)=1$, $\forall~ \underline{\theta} \in  \Theta$,  and $\mathbb{P}_{\underline{\theta}}(W \geq0)=1$, $\forall~ \underline{\theta} \in  \Theta$.\par

We call the population associated with the larger location parameter $\theta_2$ the ``best" population and the population associated with the smaller location parameter $\theta_1$ the ``worst" population. In case of tie (i.e., when $\mu_1=\mu_2$), we arbitrarily tag one of the populations (say, $\Pi_1$) as the best population. For the goal of selecting the best population, consider the natural selection rule that chooses the population corresponding to the larger of the two sample minimums, $Z_2$, as the best population.\\
Let $M \equiv M(\underline{T})$ denote the index of the selected population, $i.e.,$ $M=i$, if $X_i = Z_2$, $i=1,2$. 
Following selection of the best population, we are interested in estimating the location parameter of the selected population defined by 

\begin{equation} \label{S2.E1}
\begin{split}
\mu_M &= \begin{cases}
\mu_1, & \text{if $X_1\geq X_2$ }\\
\mu_2, & \text{if $X_1 <  X_2$}
\end{cases} 
= \mu_1I(X_1\geq X_2) + \mu_2I(X_1 < X_2)
\end{split},
\end{equation}
under the scaled squared-error loss function

\begin{equation}
\label{S2.E2}
L_{\underline{T}}(\underline{\theta},a) = \left( \frac{a-\mu_M}{\sigma}\right)^2  ,~~ \underline{\theta} \in \Theta, ~ a \in \mathcal{A},
\end{equation}
where $\mathcal{A}=\mathbb{R}$ denotes the action space. Note that $\mu_{M}$ is a random parameter in the sense that, apart from the population parameters $\underline{\theta}$, it also depends on the sample statistic $\underline{X}=(X_1, X_2)$.\par

\noindent As $\underline{T}=({X}_{1}, X_{2}, S )$ is a complete and sufficient (and hence a minimal sufficient) statistic for $\underline{\theta}\in{\Theta}$, we will pay attention to only those estimators that depend on observations $X_{ij}, j=1,\ldots,n, i=1,2,$ only through $\underline{T}$. Then the estimation problem described above is equivariant under the affine group of transformations $\mathcal{G}=\{g_{a,b}: a>0, b \in {\mathbb{R}}\}$, where $g_{a,b}(x_{1}, x_{2}, s)=(ax_{1}+b, a{x}_{2}+b, a s), (x_{1}, x_{2}, s) \in \mathbb{R}^{2} \times (0,\infty), a>0, b \in {\mathbb{R}}$ and under the group of permutations $\mathcal{G}_p=\{g_1,g_2\}$, where $g_{1}(x_{1}, x_{2}, s)=(x_1,x_2,s), g_{2}(x_{1}, x_{2}, s)=(x_2,x_1,s), (x_{1}, x_{2}, s) \in \mathbb{R}^{2} \times (0,\infty)$. The  principle of equivariance requires that we restrict our attention to only affine and permutation equivariant estimators. Any such estimator will be of the form
\begin{equation} \label{S2,E3}
\delta_{\Psi}(\underline{T})= Z_1 - S\Psi\left(W\right),
\end{equation}
for some function $\Psi:[0, \infty)\rightarrow \mathbb{R}$. Let $\mathcal{B}_1$ denote the class of all affine and permutation equivariant estimators of the type (2.3).
 An estimator $\delta(\underline{T})$ is said to be an unbiased estimator of $\mu_M$ if
 \begin{equation*}
 \mathbb{E}_{\underline{\theta}}\left(\delta(\underline{T})-\mu_M\right)=0,~~ \forall ~\underline{\theta} ~\in\Theta,
 \end{equation*}
 $i.e.,$ if $\delta(\underline{T})$ is an unbiased estimator of $\mathbb{E}_{\underline{\theta}}\left(\mu_M\right)$.\par
 An estimator $\delta^*$ is said to be the uniformly minimum variance unbiased estimator(UMVUE) of $\mu_{M}$ if among all unbiased estimators of $\mu_{M}$ it has smallest variance, uniformly. In Section 3 of the paper we derive the UMVUE of $\mu_{M}$. \par 
 The risk function (also referred to as the scaled mean squared error) of an estimator $\delta(\underline{T})$ (not necessarily belonging to $\mathcal{B}_{1}$) is given by 
 \begin{align}
 R(\underline{\theta},\delta)& =  \mathbb{E}_{\underline{\theta}}\left(L_{\underline{T}}(\underline{\theta},\delta(\underline{T}))\right) \nonumber \\
 & =\mathbb{E}_{\underline{\theta}}\left(\left(\frac{\delta(\underline{T})-\mu_M}{\sigma}\right)^2\right),~~\underline{\theta} \in \Theta.
 \label{eq:1}
 \end{align}
  Clearly  the scaled mean squared error of an estimator $\delta_{\Psi}(\underline{T})\in \mathcal{B}_{1}$, depend on $\underline{\theta}$ only through $\mu=\frac{n(\theta_{2}-\theta_{1})}{\sigma}$. We, therefore, denote the scaled mean squared error of an estimator $\delta_{\Psi}(\underline{T})\in \mathcal{B}_{1}$, by $R_{\mu}(\delta_{\Psi})$, $\mu\geq0.$\\
 A naive/natural estimator of $\mu_{M}$ can be obtained by replacing $\mu_{1}$ and $\mu_{2}$ in the definition of  $\mu_M$ (given by \eqref{S2.E1}) by their best affine equivariant estimators (BAEEs) $\widehat{\mu}_{1,1}= X_1-\frac{S}{n(2n-1)}$ and $\widehat{\mu}_{2,1}= X_2-\frac{S}{n(2n-1)}$. This yields the natural estimator $\delta_{k_2}(\underline{T})=Z_{2}-\frac{S}{n(2n-1)}$. Clearly $\delta_{k_2}\in \mathcal{B}_{1}$. We also consider a subclass $\mathcal{B}_{2}=\{\delta_{{c}_{n}}:c_{n}\in \mathbb{R}\}$ of affine and permutation equivariant  estimators, where $\delta_{{c}_{n}}(\underline{T})=Z_{2}-c_{n}S= Z_1-S(c_n-W), c_{n}\in \mathbb{R}$.
 The class $\mathcal{B}_{2}$ seems to be a natural class of estimators for estimating $\mu_{M}$. We call the class $\mathcal{B}_{2}$, the class of linear, affine and permutation equivariant estimators. Let $k_1=\frac{1}{2n(n-1)}$ and  $k_2=\frac{1}{n(2n-1)}$. Then, the estimators $\delta_{0}(\underline{T})=Z_2$, $\delta_{k_1}(\underline{T})=Z_2- k_1S$, and $\delta_{k_2}(\underline{T})=Z_2- k_2S$ respectively, are analogs of the maximum likelihood estimators (MLEs), the uniformly minimum variance unbiased estimators (UMVUEs) and the BAEEs of $\mu_{1}$ and $\mu_{2}$. In Section 4, we consider the criterion of scaled mean squared error and characterize admissible and inadmissible estimators within the class $\mathcal{B}_2$. In Section 4 of the paper we also derive a sufficient condition for inadmissibility of any affine and permutation equivariant estimator of the type (2.3). In such cases we also provide dominating estimators. In Section 5 we consider estimation after selection of the worst population and derive results similar to the ones derived in Section 4. A simulation study on performances of various competing estimators is reported in Section 6 of the paper.
 \section{The UMVUE and a generalized Bayes estimator of $\mu_{M}$}
 For the case, when the common scale parameter $\sigma$ is known, \cite{misra1993umvue} provided the UMVUE of $\mu_{M}$, for estimation after selection involving $k~(\geq2)$ exponential populations having unknown location parameters and a common known scale parameter. Following \cite{misra1993umvue}, the UMVUE of the random parameter $\mu_{M}$, for $k=2$ populations, is given as follows 
 \begin{equation}
 \delta_{0,U}(\underline{X})=Z_2-\frac{\sigma}{n}-\frac{\sigma}{n}e^{-\frac{n}{\sigma}(Z_2-Z_1)}.
 \end{equation}
 In the following theorem, we provide the UMVUE of $\mu_{M}$ for the case of unknown scale parameter $\sigma$.
 \begin{theorem}
 	The UMVUE of  $\mu_M$ is given by
 	\begin{equation*}\label{S3.E1}
 	\delta_U(\underline{T})=Z_2-\frac{S}{2n(n-1)}-\frac{S}{2n(n-1)}\left(1-\frac{\Delta}{S}\right)^{2(n-1)}I\left(\frac{\Delta}{S}\leq 1 \right),
 	\end{equation*}
 	where, $\Delta=n(Z_2-Z_1)$.
 \end{theorem}
 \begin{proof}
 	Using \eqref{S3.E1} and the fact that $\mathbb{E}_{\underline{\theta}}\left(\frac{S}{2(n-1)}\right)=\sigma,~\forall ~\underline{\theta} \in \Theta$, we have
 	\begin{equation}
 	\mathbb{E}_{\underline{\theta}}\left(Z_2-\frac{S}{2n(n-1)}-\frac{\sigma}{n}e^{-\frac{\Delta}{\sigma}}\right)=	\mathbb{E}_{\underline{\theta}}\left(\mu_{M}\right),~\forall~\underline{\theta} \in \Theta.
 	\end{equation}
 	Since $(X_1,X_2,S)$ is a complete-sufficient statistic, it suffices to find an unbiased estimator of $\sigma \mathbb{E}_{\underline{\theta}}\left(e^{-\frac{\Delta}{\sigma}}\right)$ based on $(Z_2,S,W)$. 
 	Since $(X_1,X_2)$ (and hence $\Delta=n(Z_2-Z_1)$) and $S$ are statistically independent, to find an unbiased estimator of $\sigma \mathbb{E}_{\underline{\theta}}\left(e^{-\frac{\Delta}{\sigma}}\right)$, it is enough to find an unbiased estimator of $\eta_{\Delta}(\sigma)=\sigma e^{-\frac{\Delta}{\sigma}}$ based on $S$, considering $\Delta$ as a fixed positive constant. Let $\nu=2(n-1)$, and let $Y_1,Y_2,\ldots,Y_{\nu}$ be a random sample from $Exp(0,\sigma),~ \sigma >0$. Then $S \stackrel{d}{=} \sum\limits_{i=1}^{\nu}Y_i$. Note that  $S \stackrel{d}{=} \sum\limits_{i=1}^{\nu}Y_i$ is a complete-sufficient statistic based on random sample $Y_1, Y_2, \ldots, Y_{\nu} $. Also, note that, for any fixed constant $\Delta >0$,
 	\begin{equation*}
 	\mathbb{E}_{\sigma}\left[Y_1 I(Y_1>\Delta)-\Delta I(Y_1>\Delta)\right]= \sigma e^{-\frac{\Delta}{\sigma}},~\forall~ \sigma >0 .
 	\end{equation*}
 	Thus, for any fixed constant $\Delta >0$, an unbiased estimator of  $\eta_{\Delta}(\sigma)=\sigma e^{-\frac{\Delta}{\sigma}}$, based on $S~ (\stackrel{d}{=} \sum\limits_{i=1}^{\nu}Y_i)$, is  
 	\begin{align*}
 	\psi_{\Delta}(S)&=\mathbb{E}_{\sigma}\left(\left(Y_1 I(Y_1>\Delta)-\Delta I(Y_1>\Delta)\right)|S\right)\\
 	&=S ~\mathbb{E}_{\sigma}\left(\frac{Y_1}{S} I\left(\frac{Y_1}{S}  > \frac{\Delta}{S}\right)\bigg| S\right)-\Delta \mathbb{E}_{\sigma}\left(I \left(\frac{Y_1}{S} > \frac{\Delta}{S} \right) \bigg| S\right).
 	\end{align*}
 	 Here $\frac{Y_1}{S} \stackrel{d}{=} \frac{Y_1}{\sum\limits_{i=1}^{\nu}Y_i} \sim Beta\left(1, \nu-1\right)$ (the Beta distribution) is an ancillary statistic and thus, by Basu's theorem, the complete sufficient statistic $S\left(=\sum\limits_{i=1}^{\nu}Y_i\right)$ and $\frac{Y_1}{S}$ are statistically independent. Consequently, 
 \begin{align}
 \psi_{\Delta}(S)&=\left[S \times (\nu-1)\displaystyle{\int_{\frac{\Delta}{S}}}^{1}t(1-t)^{\nu-2}dt-\Delta \times (\nu-1)\displaystyle{\int_{\frac{\Delta}{S}}}^{1}(1-t)^{\nu-2}dt\right]I\left(\frac{\Delta}{S} \leq 1\right) \nonumber \\ 
 &=\frac{S}{\nu}\left(1-\frac{\Delta}{S}\right)^{\nu} \times I\left(\frac{\Delta}{S}\leq 1 \right).
 \end{align}
 
Now using  (3.2) and (3.3), the UMVUE of $\mu_{M}$ is
 	
 	\begin{equation*}
 	\delta_U(\underline{T})=Z_2-\frac{S}{2n(n-1)}-\frac{S}{2n(n-1)}\left(1-\frac{\Delta}{S}\right)^{2(n-1)}I\left(\frac{\Delta}{S}\leq 1 \right).
 	\end{equation*}
 \end{proof}
 \noindent Now we will derive the generalized Bayes estimator of $\mu_{M}$ under the scaled squared error loss function \eqref{S2.E2}.\\
 Suppose that the unknown state of nature $\underline{\theta}=\left(\mu_{1},\mu_{2},\sigma\right)~\in \Theta$ is considered to be a realization of the random vector $\underline{R}=\left(R_1,R_2,R_3\right)$. We consider the non-informative prior density of $\underline{R}$, defined by,
 \begin{equation} \label{S3.E2}
 \Pi_{\underline{R}}(\mu_{1},\mu_{2},\sigma)=\frac{1}{\sigma},~~\forall ~(\mu_{1},\mu_{2},\sigma) \in \mathbb{R}^2\times \left(0,\infty\right).
 \end{equation}
 The posterior density function of $\underline{R}=\left(R_1,R_2,R_3\right)$ given $\underline{T}=(x_1,x_2,s)$ is given by 
 \begin{align*}
 \Pi_{\underline{R}|\underline{T}}(\mu_{1},\mu_{2},\sigma|(x_1,x_2,s))\propto\frac{1}{\sigma^{2n+1}}e^{-\frac{n(x_1-\mu_1)}{\sigma}}e^{-\frac{n(x_2-\mu_2)}{\sigma}}e^{-\frac{s}{\sigma}},~~\mu_{1}\leq x_1,\mu_{2}\leq x_2, \sigma >0.
 \end{align*}
 Under the scaled squared error loss function, the generalized Bayes estimator of $\mu_{M}$ is,
 \begin{align*}
 \delta_{GB}(\underline{T})&=\frac{\int_{0}^{\infty}\int_{-\infty}^{x_2} \int_{-\infty}^{x_1}\frac{\mu_{M}}{\sigma^2}\Pi_{\underline{R}|\underline{T}}(\mu_{1},\mu_{2},\sigma|\underline{x})~d\mu_{1}d\mu_{2}d\sigma}{\int_{0}^{\infty}\int_{-\infty}^{x_2} \int_{-\infty}^{x_1}\frac{1}{\sigma^2}\Pi_{\underline{R}|\underline{T}}(\mu_{1},\mu_{2},\sigma|\underline{x})~d\mu_{1}d\mu_{2}d\sigma}\\
 &= \begin{cases}
 \displaystyle 
 \frac{\int_{0}^{\infty}\int_{-\infty}^{x_2} \int_{-\infty}^{x_1}\frac{\mu_{1}}{\sigma^2}\Pi_{\underline{R}|\underline{T}}(\mu_{1},\mu_{2},\sigma|\underline{x})~d\mu_{1}d\mu_{2}d\sigma}{\int_{0}^{\infty}\int_{-\infty}^{x_2} \int_{-\infty}^{x_1}\frac{1}{\sigma^2}\Pi_{\underline{R}|\underline{T}}(\mu_{1},\mu_{2},\sigma|\underline{x})~d\mu_{1}d\mu_{2}d\sigma} & \text{if $X_1 \geq X_2$ }\\\\ 
 \displaystyle
 \frac{\int_{0}^{\infty}\int_{-\infty}^{x_2} \int_{-\infty}^{x_1}\frac{\mu_{2}}{\sigma^2}\Pi_{\underline{R}|\underline{T}}(\mu_{1},\mu_{2},\sigma|\underline{x})~d\mu_{1}d\mu_{2}d\sigma}{\int_{0}^{\infty}\int_{-\infty}^{x_2} \int_{-\infty}^{x_1}\frac{1}{\sigma^2}\Pi_{\underline{R}|\underline{T}}(\mu_{1},\mu_{2},\sigma|\underline{x})~d\mu_{1}d\mu_{2}d\sigma} & \text{if $X_1 <  X_2$}
 \end{cases} \\
 &=\begin{cases}
 X_1-\frac{S}{n(2n-1)} & \text{if $X_1 \geq X_2$ }\\
 X_2-\frac{S}{n(2n-1)} & \text{if $X_1 <  X_2$}
 \end{cases} \\
 &=Z_2-\frac{1}{n(2n-1)}S= \delta_{{k}_{2}}\left(\underline{T}\right), 
\end{align*}
 where, $k_2=\frac{1}{n(2n-1)}$ and $\delta_{{k}_{2}} \in \mathcal{B}_2$ is the analog of the BAEEs of $\mu_{1}$ and $\mu_{2}$. Thus, we have the following result.
 \begin{theorem}
 	Under the scaled squared error loss function $\eqref{S2.E2} $, the natural estimator $\delta_{k_2}\left(\underline{T}\right)$ (analogue of the BAEEs of $\mu_{1}$ and $\mu_{2}$) is the generalized
 	Bayes estimator of $\mu_{M}$, with respect to the non-informative prior distribution given by \eqref{S3.E2}.
 \end{theorem}
In the following Section we consider the scaled squared error loss function (2.2) and characterize admissible and inadmissible estimators, in the class $\mathcal{B}_2$ of linear, affine and permutation equivariant estimators.
 \section{Admissibility results under mean squared error criterion and natural minimax estimator of $\mu_{M}$}
 For the estimation of location parameter $\mu_{M}$ of the selected population, a natural class of estimators is the class $\mathcal{B}_2= \{\delta_{{c}_{n}}:c_n \in \mathbb{R}\}$ of linear, affine and permutation equivariant estimators; here $\delta_{{c}_{n}}(\underline{T})=Z_2-c_nS, c_n \in \mathbb{R}$. This class of estimators contains three natural estimators $\delta_{0}(\underline{T})=Z_2$, $\delta_{k_1}(\underline{T})=Z_2-\frac{S}{2n(n-1)}$ and $\delta_{k_2}(\underline{T})=Z_2-\frac{S}{n(2n-1)}$, which are, respectively, analogues of the MLEs, UMVUEs and BAEEs of $\mu_{1}$ and $\mu_{2}$. Thus, it is pertinent to study properties of estimators belonging to class $\mathcal{B}_2$.
 The following lemma will be  useful in obtaining the admissible and minimax estimators in class $\mathcal{B}_2$ under the criterion of scaled mean squared error (see (2.4)).
 
 \begin{lemma}
 	Let  $U=\frac{Z_2-\mu_{M}}{\sigma}$. Then, for any $\underline{\theta} \in \Theta$,
 	\begin{itemize}
 		\item [(i)] $\mathbb{E}_{\underline{\theta}}(U)= \frac{1}{n}\left[ \left(\frac{\mu+1}{2}\right)e^{-\mu}+1\right]$, $\mu \geq 0$;
 		\item [(ii)] $\mathbb{E}_{\underline{\theta}}(U^2)= \frac{1}{n^2}\left[\frac{(\mu^2+3\mu+3)}{2}e^{-\mu}+2\right]$, $\mu \geq 0$.
 	\end{itemize}
 \end{lemma}
\begin{proof}
	Let $Y_i=\frac{n}{\sigma}(X_i-\mu_i)~i=1,2$, so that $Y_1$ and $Y_2$  are $ Exp\left(0,1\right)$. Note that, the distribution of $U$ is a permutation symmetric function of $(\mu_{1},\mu_{2})$. Thus, without loss of generality, we may take $\mu_{i}=\theta_i, i=1,2$. Then,
	\begin{align*}
	\mathbb{E}_{\underline{\theta}}(U)&=\mathbb{E}_{\underline{\theta}}\left(\frac{Z_2-\mu_{M}}{\sigma}\right)\\
	&=\mathbb{E}_{\underline{\theta}}\left(\frac{X_1-\mu_{1}}{\sigma}I(X_1\geq X_2)\right)+\mathbb{E}_{\underline{\theta}}\left(\frac{X_2-\mu_{2}}{\sigma}I(X_1 < X_2)\right)\\
	&=\frac{1}{n}\left[\mathbb{E}_{\underline{\theta}}\left(Y_1I(Y_1-Y_2\geq \mu)\right)+\mathbb{E}_{\underline{\theta}}\left(Y_2I(Y_2-Y_1 > -\mu )\right)\right]\\
	&=\frac{1}{n}\left[ \displaystyle\int_{0}^{\infty} \displaystyle\int_{y_2+\mu}^{\infty} y_1 e^{-y_1} e^{-y_2} dy_1 dy_2  + \displaystyle\int_{0}^{\infty} \displaystyle\int_{\max\{0,y_1-\mu\}}^{\infty} y_2 e^{-y_2} e^{-y_1} dy_2 dy_1\right]\\
&= \frac{1}{n}\left[ \left(\frac{\mu+1}{2}\right)e^{-\mu}+1\right].
\end{align*}
Similarly for any $\underline{\theta} \in \Theta,$
\begin{align*}
\mathbb{E}_{\underline{\theta}}(U^2)&=\mathbb{E}_{\underline{\theta}}\left(\left(\frac{Z_2-\mu_{M}}{\sigma}\right)^2\right)\\
&=\frac{1}{n^2}\left[\displaystyle\int_{0}^{\infty} \displaystyle\int_{y_2+\mu}^{\infty} y_1^2 e^{-(y_1+y_2)} dy_1 dy_2 +
\displaystyle\int_{0}^{\infty} \displaystyle\int_{\max \{0,y_1-\mu \}}^{\infty} y_2^2 e^{-(y_1+y_2)} dy_2 dy_1
\right]\\
&=\frac{1}{n^2}\left[\frac{(\mu^2+3\mu+3)}{2}e^{-\mu}+2\right].
\end{align*}
\end{proof}

\par Next, we provide a  result characterizing admissible/inadmissible estimators within the class $\mathcal{B}_2$, under the criterion of the scaled mean squared error. We also derive the minimax estimator within the class $\mathcal{B}_2$ with respect to the criterion of the scaled mean squared error.\par 
\noindent Using Lemma 4.1 and the facts that $V=\frac{S}{\sigma} \sim Gamma(2(n-1),1)$  and, $U$ and $V$ are independently distributed, the risk function of $\delta_{{c}_{n}} \in \mathcal{B}_2$ is given by
 \begin{align}\label{S4.E1}
 R_{\mu}\left(\delta_{{c}_{n}}\right)&=\mathbb{E}_{\underline{\theta}}\left[\left(\frac{Z_2-c_nS-\mu_{M}}{\sigma}\right)^2\right] \nonumber\\
 &=\mathbb{E}_{\underline{\theta}}(U^2)-2c_n\mathbb{E}_{\underline{\theta}}\left(U\right)\mathbb{E}_{\underline{\theta}}\left(V\right) +c_n^2 \mathbb{E}_{\underline{\eta}}\left(V^2\right)\nonumber\\
 &=\frac{1}{n^2}\left[\frac{(\mu^2+3\mu+3)}{2}e^{-\mu}+2\right] - \frac{ 4(n-1)}{n}\left[\left(\frac{\mu+1}{2}\right)e^{-\mu}+1\right]c_n \nonumber																																		\\ 
 & ~~~~~~~~~~+ 2(n-1)(2n-1)c_n^2  ,~~~\mu \geq 0.
 \end{align}
 For any fixed $\mu \geq 0$, the risk function $R_{\mu}\left(\delta_{{c}_{n}}\right)$ is strictly bowl-shaped with unique minimum at $c_n=c_n^*(\mu)$, where,
 \begin{equation}\label{S4.E2}
 c_n^*(\mu)=\frac{1}{n(2n-1)}\left[\left(\frac{\mu+1}{2}\right)e^{-\mu}+1\right],~~~\mu \geq 0.
 \end{equation}
 Clearly
 \begin{equation}\label{S4.E3}
 \inf_{\mu\geq  0} c_n^*(\mu) = \lim_{{\mu\rightarrow \infty}} c_n^*(\mu)= \frac{1}{n(2n-1)} =k_2~ (\text{say}),~~  
 \sup_{\mu \geq  0} c_n^*(\mu)= c_n^*(0)= \frac{3}{2n(2n-1)}=k_3~(\text{say}).
 \end{equation}
 Using (\ref{S4.E1})-(\ref{S4.E3}), we have the following theorem that characterizes all admissible estimators within the class $\mathcal{B}_2$ of linear, affine and permutation equivariant estimators. The theorem additionally presents restricted minimax estimator within the class $\mathcal{B}_2$ under the scaled squared error loss function (\ref{S2.E2}).
 \begin{theorem}
 Recall that  $k_2= \frac{1}{n(2n-1)}$ and $k_3= \frac{3}{2n(2n-1)}$. For the problem of estimating $\mu_{M}$, consider the criterion of scaled mean squared error.
 	\begin{enumerate}
 		\item  [(a)] The estimators in the class $\mathcal{B}_{2,M} = \Big\{\delta_{{c}_{n}} \in \mathcal{B}_2: c_n \in \left[k_2, k_3 \right]\Big\}$ are admissible among the estimators in the class $\mathcal{B}_2$. Moreover, the estimators in the class $\mathcal{B}_{2,1}=\Big\{\delta_{{c}_{n}}: c_n \in \left(-\infty, k_2\right)\cup\left(k_3, \infty\right)\Big\}$ are inadmissible. For any $-\infty < b_n < c_n \leq k_2$ or $k_3 \leq c_n <b_n<\infty$,
 		$$R_{\mu}\left(\delta_{{c}_{n}}\right)< R_{\mu}\left(\delta_{{b}_{n}}\right),~~~ \forall~\mu \geq 0.$$
 		\item  [(b)] Let $ r_n \in \left[k_2,k_3\right]$ be such that $$4(n-1)(2n-1)r_n-8(n-1)^2r_n e^{-(4n(n-1)r_n-1)}-\frac{4(n-1)}{n}=0.$$ Then the estimator $\delta_{{r}_{n}}$ is minimax among the estimators in the class $\mathcal{B}_2$ of linear, affine and permutation equivariant estimators, i.e. $ \inf\limits_{ \delta \in  \mathcal{B}_2 } \sup\limits_{ \mu \geq 0 }R_{\mu}\left(\delta_{{c}_{n}}\right)= \sup\limits_{ \mu \geq 0 }R_{\mu}\left(\delta_{r_n}\right)$.
 	\end{enumerate}
 \end{theorem}
 \begin{proof}
 	$(a)$  Let $\mu \geq 0$ be fixed. It is clear from $(4.1)$ that $R_{\mu}(\delta_{{c}_{n}})$ is strictly decreasing for $c_n < c_n^*(\mu)$ and strictly increasing for $c_n > c_n^*(\mu)$, where  $ c_n^*(\mu)$ is defined by $(4.2)$. Since  $k_2 <c_n^*(\mu)\leq k_3,~\forall \mu \geq 0$ (see $(4.3)$), it follows that, for any $\mu \geq  0$,  $R_{\mu}(\delta_{{c}_{n}})$ is strictly decreasing for $c_n \leq k_2$ and $R_{\mu}(\delta_{{c}_{n}})$ is strictly increasing for $c_n \geq k_3$. This proves the second assertion of $(a)$. To prove the first assertion of $(a)$, note that, for any $\mu \geq 0$, $R_{\mu}(\delta_{{c}_{n}})$ is uniquely minimized at $c_n=c_n^*(\mu)$. Since  $c_n^*(\mu)$ is a continuous function of $\mu \in [0,\infty)$, it follows   $c_n^*(\mu)$ takes all values in the interval $\left(k_2, k_3 \right]$. This proves that all estimators in the class $\Big\{\delta_{{c}_{n}}:  k_2 < c_n \leq k_3 \Big\}$ uniquely minimize the risk $R_{\mu}(\delta_{{c}_{n}})$ at some $\mu \in \left[0, \infty\right)$, which implies that all the estimators in the class $\Big\{\delta_{{c}_{n}}: k_2 < c_n \leq k_3 \Big\}$ are admissible among the estimators in the class $\mathcal{B}_2$. Now, in order to complete the proof of the theorem, it requires to show that the estimator $\delta_{k_2}$ is admissible in $\mathcal{B}_2$. To prove this, let $\delta_{{c}_{n}}$ is an estimator in $\mathcal{B}_2$ such that 
 	\begin{equation}
 	R_{\mu}\left(\delta_{{c}_{n}}\right) \leq R_{\mu}\left(\delta_{k_2}\right), \forall \mu \geq 0.
 	\end{equation}
 	$$ \Rightarrow \frac{2}{n}\Biggl\{\left(\frac{\mu+1}{2}\right)e^{-\mu}+1\Biggr\}(c_n-k_2)-(2n-1)(c_n^2-k_2^2) \geq 0,~~ \forall \mu \geq 0.$$
 By separately considering the cases $c_n < k_2$ and $c_n > k_2$, one may easily check that the above inequality can not hold for every $\mu \geq 0$ if either $c_n < k_2$ or $c_n > k_2$. Thus for the inequality (4.4) to hold we must have $c_n=k_2$ (i.e. $\delta_{{c}_{n}}\equiv \delta_{{k}_{2}}$).   ~~~~ \vspace*{3mm}\\
 	$(b)$ For any $c_n \in \mathbb{R}$ and $\mu \geq 0$, from (4.1), we have
 	\begin{align}
 	R_{\mu}(\delta_{{c}_{n}})&=\frac{e^{-\mu}}{2n^2}\left[(\mu^2+3\mu+3)-4n(n-1)(\mu+1)c_n\right]+2(n-1)(2n-1)c_n^2-\frac{4(n-1)}{n}c_n+\frac{2}{n^2}.
 	\end{align}
 	In the light of Theorem 4.1.(a), it is enough to find the minimax estimator in the subclass $\mathcal{B}_{2,M}$. 
 	We have, from (4.5),
 	\begin{equation*}
 \frac{\partial }{\partial \mu}R_{\mu}(\delta_{{c}_{n}})=-\frac{\mu e^{-\mu}}{2n^2}\left[\mu-(4n(n-1)c_n-1)\right],~\mu \geq 0.
 	\end{equation*}
 	Thus, for any $c_n \in  \left[ k_2, k_3 \right] $ (so that $\mu_0=4n(n-1)c_n-1 \geq 0$), $R_{\mu}(\delta_{{c}_{n}})$ attains its supremum at $\mu=\mu_0$, i.e.,
 	\begin{align}
 		\sup\limits_{ \mu \geq 0 }R_{\mu}\left(\delta_{{c}_{n}}\right)&=R_{\mu_0}\left(\delta_{{c}_{n}}\right) \nonumber \\
 		&=\frac{1}{2n^2}e^{-(4n(n-1)c_n-1)}\left[4n(n-1)c_n+1\right]+2(n-1)(2n-1)c_n^2-\frac{4(n-1)}{n}c_n+\frac{2}{n^2} \nonumber \\
 		&= \Psi_1(c_n)~~ (\text{say}),~ c_n \in \left[k_2,k_3\right].
 	\end{align}
 We have, for any $c_n \in  \left[ k_2, k_3 \right]$,
 \begin{equation}
 \frac{d}{d c_n}\Psi_1(c_n)= -8(n-1)^2c_ne^{-(4n(n-1)c_n-1)}+4(n-1)(2n-1)c_n-\frac{4(n-1)}{n}
 \end{equation}and
 \begin{align}
 \frac{d^2}{d c_n^2}\Psi_1(c_n)&= 8(n-1)^2(4n(n-1)c_n-1)e^{-(4n(n-1)c_n-1)}+4(n-1)(2n-1) \nonumber \\
 & >0,
 \end{align}
 as $c_n \geq k_2$ implies that $c_n \geq \frac{1}{4n(n-1)}$ ($k_2 > \frac{1}{4n(n-1)}$).
 Also, note that
 \begin{equation}
 \left[ \frac{d}{d c_n}\Psi_1(c_n)\right]_{c_n=k_2}=-\frac{8(n-1)^2}{n(2n-1)}e^{-\frac{2n-3}{2n-1}} < 0
 \end{equation}
 and 
 \begin{equation}
 \left[\frac{d}{d c_n}\Psi_1(c_n)\right]_{c_n=k_3}=\frac{2(n-1)}{n}e^{-\frac{4n-5}{2n-1}}\left[e^{\frac{4n-5}{2n-1}}-\frac{6(n-1)}{2n-1}\right] > 0,
 \end{equation}
 as $e^{x} > 1+x,~ \forall x \in \mathbb{R}$. 
 Let $r_n \in \left[k_2, k_3\right]$ be the root of the equation $\frac{d}{d c_n}\Psi_1(c_n)=0$. Then, using $(4.6)$-$(4.10)$, it follows that, for $c_n \in \left[k_2, k_3\right]$, $\sup\limits_{ \mu \geq 0 }R_{\mu}\left(\delta_{{c}_{n}}\right)$ is minimized at $c_n=r_n$. Hence the assertion follows.
 \end{proof}
\noindent The observations made in the following remark are noteworthy.
 \begin{remark}
 	\begin{enumerate}
	\item [(a)] Let us define the scaled bias of an estimator $\delta_{{c}_{n}} \in \mathcal{B}_2$ of $\mu_{M}$ as $B_{\mu}\left(\delta_{{c}_{n}}\right)= \frac{1}{\sigma} \mathbb{E}_{\underline{\eta}}\left(\delta_{{c}_{n}}-\mu_{M}\right),~ \mu \geq 0$. Using Lemma $4.1 (i)$, the scaled bias of the estimators $\delta_{{c}_{n}} \in \mathcal{B}_2$ is given by 
 		\begin{align}
 		 	B_{\mu}\left(\delta_{{c}_{n}}\right) &= \frac{1}{\sigma} \left[ \mathbb{E}_{\underline{\eta}}\left(Z_2-c_nS-\mu_{M}\right) \right] \nonumber\\
 			&=\mathbb{E}_{\underline{\eta}}(U)- c_n \mathbb{E}(V) \nonumber\\
 			&= \frac{1}{n}\left[\left(\frac{\mu+1}{2}\right)e^{-\mu}+1\right]-2(n-1)c_n .
 		\end{align}
 It is evident from $(4.11)$ that the estimator $\delta_{{c}_{n}} \in \mathcal{B}_2$ is asymptotically unbiased, i.e. $$ \lim_{{n\rightarrow \infty}} B_{\mu}\left(\delta_{{c}_{n}}\right) = 0, \forall~ \mu ~ \geq 0, $$ provided that ~$\lim_{{n\rightarrow \infty}} nc_n=0$.
 \item [(b)] On account of Theorem 4.1.(a), we conclude that the estimator $\delta_{0}(\underline{T})=Z_2$ is inadmissible under the scaled  mean squared error criterion and the estimators $\delta_{k_1}(\underline{T})=Z_2-\frac{S}{2n(n-1)}$ and $\delta_{k_2}(\underline{T})=Z_2-\frac{S}{n(2n-1)}$ are admissible in class $\mathcal{B}_2$. 
 \item [(c)] From $(4.5)$, it is readily apparent that~ $\lim_{{n\rightarrow \infty}} B_{\mu}(\delta_{{c}_{n}}) =0$ and  ~ $\lim_{{n\rightarrow \infty}} R_{\mu}(\delta_{{c}_{n}}) =0, \forall~ \mu \geq 0$, provided that $\lim_{{n\rightarrow \infty}} nc_n  =0$. Therefore, any estimator $\delta_{{c}_{n}} \in \mathcal{B}_2$, satisfying  ~$\lim_{{n\rightarrow \infty}} nc_n  =0$, is consistent for estimating $\mu_{M}$ $\left( i.e.~ \delta_{{c}_{n}}(\underline{T})-\mu_{M} \overset{P}{\to} 0, as ~n \rightarrow \infty\right)$.
 	\end{enumerate}
 \end{remark}

 \subsection{A sufficient condition for inadmissibility }
For estimation of random parameter $\mu_{M}$ under the scaled squared error loss function \eqref{S2.E2}, any affine and permutation equivariant estimator has the form $  \delta_{\Psi}(\underline{T})=Z_1-S\Psi\left(W\right)$, for some function  $ \Psi : \left[0,\infty \right) \rightarrow \mathbb{R}$, where $W=\frac{Z_2-Z_1}{S}$.  We will apply the Brewster-Zidek technique (\cite{brewster1974improving}) to derive a sufficient condition for inadmissibility of any arbitrary affine  and permutation equivariant estimator 
 \begin{equation}
 \delta_{\Psi}(\underline{T})=Z_1-S\Psi\left(W\right).
 \end{equation}
 Note that any estimator $\delta_{{c}_{n}} \in \mathcal{B}_2$ (the class of linear, affine and permutation equivariant estimators) is 
of the form (4.12) with $ \Psi(W) \equiv \Psi_{c_n}(W)=c_n-W $.  \par
Recall that  $Z=Z_2-Z_1$, $W=\frac{Z}{S}$ and $V=\frac{S}{\sigma}$. Let $U'=\frac{Z_1-\mu_{M}}{\sigma}$, $f_{1,\underline{\theta}}(\cdot)$ denote the p.d.f. of $W,$ and let $f_2(\cdot)$ denote the p.d.f. of $V \sim Gamma(2(n-1),1)$. The following lemmas will be useful in proving the main result of this section.

\begin{lemma} \label{S4,L2}
	Let $w \in \left(0,\infty\right)$ be a fixed constant. Then, the conditional p.d.f. of $(U',V)$, given $W=w$, is given by
	

	\begin{equation*}
	f_{3, \underline{\theta}}(u,v|w)=\begin{cases*}
	\frac{n^2}{\Gamma(2(n-1))f_{1, \underline{\theta}}(w)}v^{(2n-2)}e^{-v\left(1+nw\right)}e^{-2nu}e^{-\mu}, & \parbox[t]{6cm}{$-\frac{\mu}{n}<u<0,~~-\frac{u}{w}<v<\infty$\\ \hspace*{2cm}or,\\ $0<u< \frac{\mu}{n},~~0<v< \infty $}\\
	\frac{n^2}{\Gamma(2(n-1))f_{1, \underline{\theta}}(w)}v^{(2n-2)}e^{-v\left(1+nw \right)}e^{-2nu}\left(e^{-\mu}+e^{\mu} \right), & \parbox[t]{5.5cm}{$\frac{\mu}{n}<u<\infty,~~0<v< \infty $} \\
	\end{cases*}
	\end{equation*}

\end{lemma}
\begin{proof} Since the pdf of $(U',V)$, given $W=w$ ($w \in \left(0,\infty\right)$) is a permutation symmetric function of ($\mu_{1}$, $\mu_{2}$), without loss of generality, we may assume that $\mu_{i}=\theta_i,~i=1,2$.
 Fix $w \in \left(0,\infty\right)$. Then the cumulative distribution function (c.d.f.) of $(U',V)$, given $W=w$, is 
\begin{align*}
	F_{3,\underline{\theta}}(u,v|w)&=\mathbb{P}_{\underline{\theta}}\left(U'\leq u, V\leq v |~W=w\right) \\
	&= \frac{1}{f_{1,\underline{\theta}}(w)}\lim_{h \downarrow 0}\frac{N(h|u,v,w,\underline{\theta})}{h},~~  -\infty <u < \infty, ~v >0,
\end{align*}
	where, for $-\infty < u < \infty$, $v >0$ and  $h >0$ (sufficiently small)
\begin{align*}
	N(h|u,v,w,\underline{\theta})&= \mathbb{P}_{\underline{\theta}}\left(U'\leq u, V\leq v , w-h< \frac{Z_2-Z_1}{S} \leq w\right)
	\\&=  \mathbb{P}_{\underline{\theta}}\left(X_1 < X_2,\frac{X_1 - \mu_2}{\sigma}\leq u, V \leq v, (w-h) < \frac{(X_2 - X_1)}{S}\leq w\right)\\ 	&\hspace{3mm}+ \mathbb{P}_{\underline{\theta}}\left(X_2 < X_1,\frac{{X_2} - \mu_1}{\sigma}\leq u,V \leq v, (w-h) < \frac{(X_1 - X_2)}{S}\leq w\right)\\
	&=	\mathbb{P}_{\underline{\theta}}\left(\frac{X_1 - \mu_2}{\sigma}\leq u,V \leq v, (w-h)V < \frac{(X_2 - X_1)}{\sigma}\leq w V\right)\\ 	&\hspace{3mm}+ \mathbb{P}_{\underline{\theta}}\left(\frac{{X_2} - \mu_1}{\sigma}\leq u,V \leq v, (w-h)V < \frac{(X_1 - X_2)}{\sigma}\leq w V\right)\\
	&= \int_{0}^{v} \left[\mathbb{P}_{\underline{\theta}}\left(\frac{{X_1} - \mu_2}{\sigma}\leq u, (w-h)t < \frac{({X_2} - {X_1})}{\sigma}\leq wt\right)\right.  \\
	& \left. \hspace{2cm}+~ \mathbb{P}_{\underline{\theta}}\left(\frac{{X_2} - \mu_1}{\sigma}\leq u, (w-h)t< \frac{({X_1} - {X_2})}{\sigma}\leq wt\right)\right]f_2(t)~dt\\
	&= \int_{0}^{v} \left[\mathbb{P}_{\underline{\theta}}\left(Y_1 \leq nu + \mu, n(w-h)t+Y_1-\mu < Y_2 \leq nwt +Y_1 -\mu\right)\right.  \\
	& \left. \hspace{2cm}+~ \mathbb{P}_{\underline{\theta}}\left(Y_2 \leq nu - \mu, n(w-h)t+Y_2+\mu < Y_1 \leq nwt +Y_2 + \mu \right)\right]f_2(t)~dt,
\end{align*}
	where, $Y_i = \frac{n({X_i}-\mu_i)}{\sigma},~~ i=1,2$, so that $Y_1$ and $Y_2$ are iid  $ Exp(0,1)$. Let $h_1(\cdot)$ denote the pdf of $ Exp(0,1)$. Then for fixed $w \in (0, \infty)$,
	\begin{equation}
		F_{3,\underline{\theta}}(u,v|w)=\frac{1}{f_{1,\underline{\theta}}(w)} \left[ g_1(u,v|w, \underline{\theta}) + g_2(u,v|w, \underline{\theta})\right],~~ -\infty <u < \infty, ~v >0,
	\end{equation}
	where,
\begin{align*}
g_1(u,v|w, \underline{\theta})&=\lim_{h \downarrow 0}\frac{1}{h} \int_{0}^{v} \int_{0}^{nu+\mu} \mathbb{P}_{\underline{\theta}}\left(n(w-h)t+y-\mu < Y_2 \leq nwt +y -\mu\right)e^{-y}. \frac{e^{-t}t^{2n-3}}{\Gamma(2(n-1))} dy dt, \\
&= \int_{0}^{v} \int_{0}^{nu+\mu} nt~ h_1(nwt+y-\mu)e^{-y}\frac{e^{-t}t^{2n-3}}{\Gamma(2(n-1))} dy dt,~~ u >-\frac{\mu}{n}, v>0 
\end{align*}

\begin{align}
 \frac{\partial^2  g_1(u,v|w, \underline{\theta})}{\partial u \partial v }&= \frac{n^2}{\Gamma(2(n-1)) f_{1, \underline{\theta}}(w)}v^{(2n-2)}e^{-v\left(1+nw\right)}e^{-2nu}e^{-\mu},~~ u >-\frac{\mu}{n}, v > \max\left\{0,-\frac{u}{w}\right\}
\end{align}

\begin{align*}
g_2(u,v|w, \underline{\theta})&=\lim_{h \downarrow 0}\frac{1}{h} \int_{0}^{v} \int_{0}^{nu-\mu} \mathbb{P}_{\underline{\theta}}\left(n(w-h)t+y+\mu < Y_1 \leq nwt +y + \mu\right)e^{-y}. \frac{e^{-t}t^{2n-3}}{\Gamma(2(n-1))} dy dt, \\
&= \int_{0}^{v} \int_{0}^{nu-\mu} nt~ h_1(nwt+y+\mu)e^{-y}\frac{e^{-t}t^{2n-3}}{\Gamma(2(n-1))} dy dt,~~ u > \frac{\mu}{n}, v > 0
\end{align*}

\begin{align}
\frac{\partial^2  g_2(u,v|w, \underline{\theta})}{\partial u \partial v }&= \frac{n^2}{\Gamma(2(n-1)) f_{1, \underline{\theta}}(w)}v^{(2n-2)}e^{-v\left(1+nw\right)}e^{-2nu}e^{\mu},~~ u > \frac{\mu}{n}, v > 0.
\end{align}	
	Now the assertion follows on using $(4.13)-(4.15)$.
\end{proof}

\begin{lemma}
For $\alpha > 0 $, let  $\overline{G}_{\alpha}(x)=\displaystyle{\int_{x}^{\infty}} \frac{1}{\Gamma(\alpha)} e^{-t} t^{\alpha-1} dt,~ x \geq 0$. Define 
$$ I_1(\mu)= \int_{0}^{\frac{\mu}{n}} u e^{2nu} \overline{G}_{2n}\left(\frac{u(1+nw)}{w}\right) du, ~\mu \geq 0,$$ 
$$ I_2(\mu)= \int_{0}^{\frac{\mu}{n}}  e^{2nu} \overline{G}_{2n+1}\left(\frac{u(1+nw)}{w}\right) du, ~\mu \geq 0,$$ 
$$I_3(\mu)=\frac{\mu}{1+n I_2(\mu)},~~~I_4(\mu)=\frac{I_1({\mu})}{1+n I_2(\mu)}~~ \text{and}~~ k(\mu)=\frac{1+\mu-2 n^2 I_1(\mu)}{1+n I_2(\mu)},~~\mu \geq 0.$$
Then
	\begin{enumerate}
		\item[(i)] 
	$	\lim\limits_{\mu \rightarrow \infty} I_1(\mu) =\begin{cases}
		\left(\frac{1+nw}{1-nw}\right)^{2n}\left[\frac{w}{1-nw}-\frac{1}{4n^2}\right]+\frac{1}{4n^2}, & \text{if ~$0 <w < \frac{1}{n}$ }\\
		\infty, & \text{if~ $w \geq \frac{1}{n} $}
		\end{cases}$
		\item[(ii)] 
		$\lim\limits_{\mu \rightarrow \infty} I_2(\mu) =\begin{cases}
		\frac{1}{2n}\left[\left(\frac{1+nw}{1-nw}\right)^{2n+1}-1\right], & \text{if $0 <w < \frac{1}{n}$ }\\
		\infty, & \text{if~ $w \geq \frac{1}{n} $}
		\end{cases}$
		\item[(iii)] 
		$\lim\limits_{\mu \rightarrow \infty} I_3(\mu) =\begin{cases}
		\infty, & \text{if~ $0 <w < \frac{1}{n}$ }\\
		0, & \text{if~ $w \geq \frac{1}{n} $}
		\end{cases}$
		
	\item[(iv)] 
$\lim\limits_{\mu \rightarrow \infty} I_4(\mu) =\begin{cases}
\frac{\lim\limits_{\mu \rightarrow \infty} I_1(\mu)}{1+n \lim\limits_{\mu \rightarrow \infty} I_2(\mu)}, & \text{if~ $0 <w < \frac{1}{n}$ }\\
\frac{2w}{1+nw}, & \text{if~ $w \geq \frac{1}{n} $}
\end{cases}$
	
	\item[(v)] 
	$\lim\limits_{\mu \rightarrow \infty} k(\mu) =\begin{cases}
	\infty, & \text{if~ $0 <w < \frac{1}{n}$ }\\
	\frac{-4n^2w}{1+nw}, & \text{if~ $w \geq \frac{1}{n} $}
	\end{cases}$
	
	\item[(vi)] 
	$\sup\limits_{\mu \geq 0} k(\mu) =\begin{cases}
	\infty, & \text{if~ $0 <w < \frac{1}{n}$ }\\
	1, & \text{if~ $w \geq \frac{1}{n} $}
	\end{cases}$
	\item[(vii)] For $w \geq \frac{1}{n}$, $\inf\limits_{\mu \geq 0} k(\mu) =  -\frac{4n^2w}{1+nw} $.  \text{In general}, $ \inf\limits_{\mu \geq 0} k(\mu) \geq  -\frac{4n^2w}{1+nw},~ w \geq 0$.	
	\end{enumerate}

\end{lemma}
\begin{proof} $(i)$ For $\mu \geq 0$,
	\begin{align*}
	I_1(\mu)&=\displaystyle{\int_{0}^{\frac{\mu}{n}}} \displaystyle{\int_{\frac{u(1+nw)}{w}}^{\infty}} u e^{2nu} \frac{e^{-t}t^{2n-1}}{\Gamma(2n)} dt du \\
	&= \displaystyle{\int_{0}^{\infty}}\frac{e^{-t}t^{2n-1}}{\Gamma(2n)} \left\{ \displaystyle{\int_{0}^{\min\left\{\frac{tw}{1+nw},\frac{\mu}{n}\right\}}}u e^{2nu}du \right\}dt.
	\end{align*}
	
	Thus, 
	\begin{align*}
	\lim\limits_{\mu \rightarrow \infty} I_1(\mu)&= \displaystyle{\int_{0}^{\infty}}\frac{e^{-t}t^{2n-1}}{\Gamma(2n)} \left\{ \displaystyle{\int_{0}^{\frac{tw}{1+nw}}}u e^{2nu}du \right\}dt\\
	&= \displaystyle{\int_{0}^{\infty}} \frac{e^{-\frac{1-nw}{1+nw}t}t^{2n-1}}{2n\Gamma(2n+1)}\left\{\frac{2ntw}{1+nw}-1\right\}dt+\frac{1}{4n^2}\\
	&=\begin{cases}
	\left(\frac{1+nw}{1-nw}\right)^{2n}\left[\frac{w}{1-nw}-\frac{1}{4n^2}\right]+\frac{1}{4n^2}, & \text{if $0 <w < \frac{1}{n}$ }\\
	\infty, & \text{if $w \geq \frac{1}{n} $}
	\end{cases}.
	\end{align*}
$(ii)$ As in $(i)$, 	
	\begin{align*}
	\lim\limits_{\mu \rightarrow \infty} I_2(\mu)&= \displaystyle{\int_{0}^{\infty}}\frac{e^{-t}t^{2n}}{\Gamma(2n+1)} \left\{ \displaystyle{\int_{0}^{\frac{tw}{1+nw}}} e^{2nu}du \right\}dt\\
	&= \frac{1}{2n \Gamma(2n+1)}\displaystyle{\int_{0}^{\infty}} e^{-\frac{1-nw}{1+nw}t}t^{2n}dt- \frac{1}{2n}\\
	&=\begin{cases}
	\frac{1}{2n}\left[\left(\frac{1+nw}{1-nw}\right)^{2n+1}-1\right], & \text{if $0 <w < \frac{1}{n}$ }\\
	\infty, & \text{if $w \geq \frac{1}{n} $}
	\end{cases}
	\end{align*}
$(iii)$ For $0 < w <\frac{1}{n}$, using $(ii)$, we get
$$\lim\limits_{\mu \rightarrow \infty} I_3(\mu)= \infty .$$
For $w \geq \frac{1}{n}$, using L'Hôpital's rule, we get
\begin{align*}
\lim\limits_{\mu \rightarrow \infty} I_3(\mu)&= \lim\limits_{\mu \rightarrow \infty}\frac{1}{nI_2'(\mu)}\\
&=\lim\limits_{\mu \rightarrow \infty} \frac{e^{-2\mu}}{\overline{G}_{2n+1}\left(\frac{\mu(1+nw)}{w}\right)}\\
&= \lim\limits_{\mu \rightarrow \infty} \frac{2e^{-2\mu}}{\frac{1}{\Gamma(2n+1)}\frac{1+nw}{nw}e^{-\frac{\mu(1+nw)}{nw}}\left(\frac{\mu(1+nw)}{nw}\right)^{2n}}\\
&=2\Gamma(2n+1)\left(\frac{nw}{1+nw}\right)^{2n+1}  \lim\limits_{\mu \rightarrow \infty} \frac{1}{\mu^{2n}e^{\frac{\mu(nw-1)}{nw}}}\\
&=0.
\end{align*}
$(iv)$ For $0 <w < \frac{1}{n}$,
$$\lim\limits_{\mu \rightarrow \infty} I_4(\mu)= \frac{\lim\limits_{\mu \rightarrow \infty} I_1(\mu)}{1+n\lim\limits_{\mu \rightarrow \infty} I_2(\mu)} .$$
For $w \geq \frac{1}{n}$, using L'Hôpital's rule, we get
\begin{align*}
\lim\limits_{\mu \rightarrow \infty} I_4(\mu)&=\lim\limits_{\mu \rightarrow \infty} \frac{I_1'(\mu)}{n I_2'(\mu)}\\
&= \lim\limits_{\mu \rightarrow \infty} \frac{\frac{\mu}{n^2}e^{2\mu}\overline{G}_{2n}\left(\frac{\mu(1+nw)}{w}\right)}{e^{2\mu}\overline{G}_{2n+1}\left(\frac{\mu(1+nw)}{w}\right)}\\
&=\frac{1}{n^2} \lim\limits_{\mu \rightarrow \infty} \frac{\mu \sum\limits_{j=0}^{2n-1} \frac{\left(\frac{\mu(1+nw)}{nw}\right)^j}{j!} }{\sum\limits_{j=0}^{2n} \frac{\left(\frac{\mu(1+nw)}{nw}\right)^j}{j!}}\\
&=\frac{2w}{1+nw}.
\end{align*}
$(v)$ Using $(ii), (iii)$ and $(iv)$, we get
\begin{align*}
\lim\limits_{\mu \rightarrow \infty} k(\mu)&= \lim\limits_{\mu \rightarrow \infty} \left[\frac{1}{1+n I_2(\mu)}+I_3(\mu)-2n^2I_4(\mu)\right]\\
&=\begin{cases}
\infty, & \text{if $0 <w < \frac{1}{n}$ }\\
\frac{-4n^2w}{1+nw}, & \text{if $w \geq \frac{1}{n} $}
\end{cases}
\end{align*}
$(vi)$ For $0<w<\frac{1}{n}$, it is clear from $(v)$ that
$$\sup\limits_{\mu \geq 0} k(\mu)= \infty .$$
Note that, $k(0)=1$. Thus, for $ w \geq \frac{1}{n}$, to show that $\sup\limits_{\mu \geq 0} k(\mu)= 1$, it suffices to show that $k(\mu) \leq 1,~ \forall ~\mu \geq 0$, i.e.
	\begin{align}
 &n \displaystyle{\int_{0}^{\frac{\mu}{n}}e^{2nu}\overline{G}_{2n+1}\left(\frac{u(1+nw)}{w}\right)du}+ 2n^2 \displaystyle{\int_{0}^{\frac{\mu}{n}}ue^{2nu}\overline{G}_{2n}\left(\frac{u(1+nw)}{w}\right)du}  -\mu   \geq 0,~ \forall~ \mu \geq 0,~ w\geq \frac{1}{n} \nonumber\\
 \text{or}, & \inf\limits_{w \geq \frac{1}{n}} \left[ n \displaystyle{\int_{0}^{\frac{\mu}{n}}e^{2nu}\overline{G}_{2n+1}\left(\frac{u(1+nw)}{w}\right)du}+ 2n^2 \displaystyle{\int_{0}^{\frac{\mu}{n}}ue^{2nu}\overline{G}_{2n}\left(\frac{u(1+nw)}{w}\right)du}  -\mu   \right] \geq 0,~ \forall~ \mu \geq 0 \nonumber\\
\text{or},~ & n \displaystyle{\int_{0}^{\frac{\mu}{n}}e^{2nu}\overline{G}_{2n+1}\left(2nu\right)du}+ 2n^2 \displaystyle{\int_{0}^{\frac{\mu}{n}}ue^{2nu}\overline{G}_{2n}\left(2nu\right)du}  -\mu \geq 0,~ \forall~ \mu \geq 0 \nonumber\\
\text{or},~ & \frac{1}{2} \displaystyle{\int_{0}^{2\mu}e^{z}\overline{G}_{2n+1}\left(z\right)du}+ \frac{1}{2} \displaystyle{\int_{0}^{2\mu}ze^{z}\overline{G}_{2n}\left(z\right)du}  -\mu \geq 0,~ \forall~ \mu \geq 0 \nonumber\\
\text{or},~ & \displaystyle{\int_{0}^{t}}e^z \left[\frac{e^{-z}z^{2n}}{(2n)!}+\overline{G}_{2n}\left(z\right) \right]dz+ \displaystyle{\int_{0}^{t}}ze^z \overline{G}_{2n}\left(z\right)dz -t \geq 0,~ \forall~ t \geq 0 \nonumber \\
\text{or},~ & \frac{t^{2n+1}}{(2n+1)!} + \displaystyle{\int_{0}^{t}}e^z(z+1) \overline{G}_{2n}\left(z\right)dz -t \geq 0,~ \forall~ t \geq 0. 
	\end{align}
	Let $$ \psi_1(t)= \frac{t^{2n+1}}{(2n+1)!} + \displaystyle{\int_{0}^{t}}e^z(z+1) \overline{G}_{2n}\left(z\right)dz -t,~ t \geq 0  .$$
	Then, for $t \geq 0$,
	\begin{align*}
	\psi_1'(t)&=\frac{t^{2n}}{(2n)!} + e^{t}(t+1) \overline{G}_{2n}\left(t\right)-1\\
	&=\frac{t^{2n}}{(2n)!} + e^{t}(t+1) \displaystyle{ \int_{t}^{\infty}} \frac{e^{-x}x^{2n-1}}{(2n-1)!}-1\\
	&=\frac{t^{2n}}{(2n)!} + (t+1) \displaystyle{ \int_{0}^{\infty}} \frac{e^{-x}(x+t)^{2n-1}}{(2n-1)!}-1,
	\end{align*}
	is an increasing function of $t$ on $\left[0, \infty \right)$. Thus,
	$$ \psi_1'(t) \geq \psi_1'(0)=0,~ \forall ~t \geq 0$$
	$$ \Rightarrow  \psi_1(t) \geq \psi_1(0)=0,~ \forall ~t \geq 0$$
	establishing the inequality (4.16). Hence the assertion follows.\\	
$(vii)$ Note that for $\Delta=\frac{1+nw}{2nw} \in \left(\frac{1}{2}, \infty\right)$ 
	\begin{align*}
	k(\mu)&=\frac{1+ \mu - 2 n^2 \displaystyle{\int_{0}^{\frac{\mu}{n}} u e^{2nu} \overline{G}_{2n}\left(\frac{u(1+nw)}{w}\right) du}}{1+n\displaystyle{\int_{0}^{\frac{\mu}{n}} e^{2nu}\overline{G}_{2n+1}\left(\frac{u(1+nw)}{w}\right)du}}\\
	&= \frac{2+ 2\mu -  \displaystyle{\int_{0}^{2 \mu} t e^{t} \overline{G}_{2n}\left(\Delta t\right) dt}}{2+\displaystyle{\int_{0}^{2\mu} e^{t}\overline{G}_{2n+1}\left(\Delta t\right)dt}} ,~ \forall~ \mu \geq 0.\\
	\end{align*}
In view of $(v)$, to prove the assertion it suffices to show that, 		
	\begin{align*}
	\frac{2+ x - \displaystyle{\int_{0}^{x} t e^{t}\overline{G}_{2n}\left(\Delta t\right)dt}}{2+\displaystyle{\int_{0}^{x} e^{t}\overline{G}_{2n+1}\left(\Delta t\right)dt}} & \geq -\frac{2n}{\Delta} , ~ \forall~ x \geq 0,~ \Delta > \frac{1}{2}
		\end{align*}
		
		or, for every fixed $\Delta > \frac{1}{2}$, $\psi_2(x) \geq 0$, $\forall x \geq 0$, where,
		$$\psi_2(x) = 2n\displaystyle{\int_{0}^{x}}e^{t} \overline{G}_{2n+1}\left(\Delta t\right) dt -\Delta \displaystyle{\int_{0}^{x}}te^{t} \overline{G}_{2n}\left(\Delta t\right) dt +\Delta x+ 2\Delta +4n,~ x \geq0  .$$
		We have
		\begin{align*}
		\psi_2'(x) &=2n e^{x}\displaystyle{\int_{\Delta x}^{\infty}} \frac{e^{-t}t^{2n}}{(2n)!}dt-\Delta x e^{x} \displaystyle{\int_{\Delta x}^{\infty}} \frac{e^{-t}t^{2n-1}}{(2n-1)!}dt + \Delta\\
		&= e^{(1-\Delta)x} \displaystyle{\int_{0}^{\infty}} \frac{e^{-z}(z+\Delta x)^{2n}}{(2n-1)!}dz - \Delta x e^{(1-\Delta)x} \displaystyle{\int_{0}^{\infty}} \frac{e^{-z}(z+\Delta x)^{2n-1}}{(2n-1)!}dz + \Delta\\
		&= \frac{e^{(1-\Delta)x}}{(2n-1)!} \displaystyle{\int_{0}^{\infty}} e^{-z}z(z+\Delta x)^{2n-1} dz + \Delta \geq 0,~ \forall x \geq 0.
		\end{align*}
		Thus, $\psi_2(x) \geq \psi_2(0)=2\Delta +4n \geq 0$. Hence the assertion follows.
\end{proof}

For a function $\Psi: [0,\infty) \rightarrow \mathbb{R}$, the risk function of any affine and permutation equivariant  estimator $\delta_{\Psi}(\underline{T})=Z_1-S\Psi(W)$ can be written as 
$$R_{\mu}(\delta_{\Phi})=\mathbb{E}_{\underline{\theta}}\left(R_1(\mu,\Psi(w))\right),~~ \mu \geq 0 ,$$
where, for any fixed $w \in [0,\infty)$,
\begin{equation} \label{S4.E17}
R_1\left(\mu, \Psi(w)\right)=\mathbb{E}_{\underline{\theta}}\left[\left(\frac{Z_1-S\Psi(w)-\mu_{M}}{\sigma}\right)^2 \bigg|W=w\right],~~\mu\geq0
\end{equation}
is the conditional risk of $\delta_{\Psi}$ given $W=w$.
We aim to find the choice of $\Psi(\cdot)$ that minimizes the conditional risk in (\ref{S4.E17}), for fixed $w \in (0,\infty)$. For any fixed $\mu \in [0,\infty)$, the choice of $\Psi$ that minimizes (\ref{S4.E17}) is obtained as
\begin{align} \label{S4.E18}
\Psi_{\mu}(w)&= \frac{\mathbb{E}_{\underline{\eta}}\left(\left(\frac{Z_1-\mu_{M}}{\sigma}\right)\frac{S}{\sigma}\bigg|W=w\right)}{\mathbb{E}_{\underline{\eta}}\left(\frac{S^2}{\sigma^2}\bigg|W=w\right)}\\
&=\frac{\mathbb{E}_{\underline{\eta}}\left(U'V\bigg|W=w\right)}{\mathbb{E}_{\underline{\eta}}\left(V^2\bigg|W=w\right)}  \nonumber\\
&=\frac{\displaystyle{\int_{-\infty}^{\infty}} \displaystyle{\int_{-\infty}^{\infty}} uv f_{3, \underline{\theta}}(u,v|w) du dv }{\displaystyle{\int_{-\infty}^{\infty}} \displaystyle{\int_{-\infty}^{\infty}} v^2 f_{3, \underline{\theta}}(u,v|w) du dv} = \frac{N_1}{D_1}~(\text{say}),
\end{align}
where,
\begin{align*}
N_1&=e^{-\mu} \displaystyle{\int_{-\frac{\mu}{n}}^{0}} u e^{-2nu}\left\{\displaystyle{\int_{-\frac{u}{w}}^{\infty}} e^{-(1+nw)v}v^{2n-1}dv\right\}du +  e^{-\mu} \displaystyle{\int_{0}^{\infty}} u e^{-2nu}\left\{\displaystyle{\int_{0}^{\infty}} e^{-(1+nw)v}v^{2n-1}dv\right\}du \\ ~~\hspace{9mm} & + e^{\mu} \displaystyle{\int_{\frac{\mu}{n}}^{\infty}} u e^{-2nu}\left\{\displaystyle{\int_{0}^{\infty}} e^{-(1+nw)v}v^{2n-1}dv\right\}du \\
&= -e^{-\mu} \displaystyle{\int_{0}^{\frac{\mu}{n}}} u e^{2nu}\left\{\displaystyle{\int_{\frac{u}{w}}^{\infty}} e^{-(1+nw)v}v^{2n-1}dv\right\}du + \frac{e^{-\mu}}{4n^2} \frac{(2n-1)!}{(1+nw)^{2n}}+ \frac{2 \mu+1}{4n^2}  \frac{(2n-1)!}{(1+nw)^{2n}} e^{-\mu}\\
&=\frac{(2n-1)!e^{-\mu}}{2n^2(1+nw)^{2n}} \left[1+ \mu - 2n^2 \displaystyle{\int_{0}^{\frac{\mu}{n}}}ue^{2nu}\overline{G}_{2n}\left(\frac{u(1+nw)}{w}\right) du\right],
\end{align*}
and 
\begin{align*}
D_1&=e^{-\mu} \displaystyle{\int_{-\frac{\mu}{n}}^{0}}  e^{-2nu}\left\{\displaystyle{\int_{-\frac{u}{w}}^{\infty}} e^{-(1+nw)v}v^{2n}dv\right\}du +  e^{-\mu} \displaystyle{\int_{0}^{\infty}}  e^{-2nu}\left\{\displaystyle{\int_{0}^{\infty}} e^{-(1+nw)v}v^{2n}dv\right\}du \\ ~~\hspace{9mm} & + e^{\mu} \displaystyle{\int_{\frac{\mu}{n}}^{\infty}}  e^{-2nu}\left\{\displaystyle{\int_{0}^{\infty}} e^{-(1+nw)v}v^{2n}dv\right\}du \\
&= e^{-\mu} \displaystyle{\int_{0}^{\frac{\mu}{n}}} e^{2nu}\left\{\displaystyle{\int_{\frac{u}{w}}^{\infty}} e^{-(1+nw)v}v^{2n}dv\right\}du + \frac{e^{-\mu}}{2n} \frac{(2n)!}{(1+nw)^{2n+1}}+ \frac{e^{-\mu}}{2n}  \frac{(2n)!}{(1+nw)^{2n+1}} \\
&=\frac{(2n)!e^{-\mu}}{n(1+nw)^{2n+1}} \left[1+ n \displaystyle{\int_{0}^{\frac{\mu}{n}}}e^{2nu}\overline{G}_{2n+1}\left(\frac{u(1+nw)}{w}\right) du\right].
\end{align*}
Therefore, for fixed $w \in (0, \infty)$
$$ \Psi_{\mu}(w)= \frac{1+nw}{4n^2}k(\mu),~ \mu \geq 0,$$
where $k(\mu)$ is as defined in Lemma 4.3. \par 
\noindent Further, using Lemma 4.3., we get
\begin{equation}
\Psi^*(w)=\sup_{\mu \geq 0}\Psi_{\mu}(w)=\begin{cases}
\infty, & \text{if $0 <w < \frac{1}{n}$ }\\
\frac{1+nw}{4n^2}, & \text{if $w \geq \frac{1}{n} $}
\end{cases}
\end{equation}
and \begin{equation} 
\Psi_*(w)=\inf_{\mu \geq 0} \Psi_{\mu}(w) \geq -w,~~ \forall ~w >0.
\end{equation}
Thus, we have the following theorem which provides a sufficient condition for inadmissibility of affine and permutation equivariant estimators of $\mu_{M}$.
\begin{theorem}
	For a given function $\Psi:[0,\infty) \to \mathbb{R}$, let  $\delta_{\Psi}\left(\underline{T}\right)=Z_1-S\Phi\left(W\right) $ be an affine and permutation equivariant estimator of $\mu_{M}$. Suppose that
	$$\mathbb{P}_{\underline{\theta}}\left[\Bigg\{\left(\Psi(W)>\left(\frac{1+nW}{4n^2}\right)\right) \bigcap\left(W \geq \frac{1}{n}\right)\Bigg\} \bigcup \Bigg\{\left(\Psi(W)< -W \right) \Bigg\} \right] >0,~ \text{for some}~ \underline{\theta} \in \Theta. $$
	Then, the estimator $\delta_{\Psi}(\cdot)$ is inadmissible for estimating $\mu_{M}$ and is dominated by
	\begin{equation}
	\delta_{\Psi}^I\left(\underline{T}\right)=\begin{cases}
		Z_1+WS, & \text{ if }~ \Psi(W) < -W~ ,\\
	Z_1-\frac{1+nW}{4n^2}S, & \text{ if }~ \Psi(W) > \frac{1+nW}{4n^2}~ \text{and} ~ W \geq \frac{1}{n},\\
	\delta_{\Psi}\left(\underline{T}\right), & \text{otherwise}
	\end{cases}
	\end{equation}
	
\end{theorem}
\begin{proof}
	Note that, for any fixed $w \in (0,\infty)$ and $\underline{\theta} \in \Theta$
	$$R_1\left(\mu, \Psi(w)\right)=\mathbb{E}_{\underline{\theta}}\left[\left(\frac{Z_1-S\Psi(w)-\mu_{M}}{\sigma}\right)^2 \bigg|W=w\right]$$
	is strictly decreasing on $ \left(-\infty, \Psi_{\mu}(w) \right) $ and strictly increasing on $ \left( \Psi_{\mu}(w), \infty \right)$, where $\Psi_{\mu}(w),~ \mu \geq 0,~ w >0,$ is defined by $(4.18)$. Using this fact along with $(4.20)$ and $(4.21)$, we have the following two observations:\\ $(i)$ for any fixed $w \in (0, \infty)$ and $\underline{\theta} \in \Theta$, $R_1\left(\mu, \Psi(w)\right)$ is strictly decreasing on $(-\infty,-w)$;\\ $(ii)$ for any $w \geq \frac{1}{n}$ and $\underline{\theta} \in \Theta$, $R_1\left(\mu, \Psi(w)\right)$ is strictly increasing on $\left(\frac{1+nw}{4n^2}, \infty \right)$. Hence the result follows.
 \end{proof}
 \noindent As a direct consequence of Theorem 4.2., we have the following result on inadmissibility of linear, affine and permutation equivariant estimators in class $\mathcal{B}_2 = \{\delta_{{c}_{n}}:c_n \in \mathbb{R}\}$, where $\delta_{{c}_{n}}(\underline{T})= Z_2 -c_nS,~c_n \in \mathbb{R}$, under the criterion of scaled mean squared error. Dominating estimators are obtained, wherever pertinent.
\begin{corollary}
 Under the scaled mean squared error criterion, if $c_n \in \left\{(-\infty, 0) \bigcup \left[\frac{n+2}{4n^2}, \infty\right) \right\}$, then the linear, affine and permutation equivariant estimator $\delta_{{c}_{n}}(\cdot) \in \mathcal{B}_2$, is inadmissible for estimating $\mu_{M}$, and is dominated by $$\delta_{{c}_{n}}^I(\underline{T})=\begin{cases}
	Z_1+WS, & \text{ if }~ c_n-W  < -W~ ,\\
	Z_1-\frac{1+nW}{4n^2}S, & \text{ if }~ \frac{1}{n} \leq W < \frac{4n}{n+1}\left(c_n-\frac{1}{4n^2}\right),\\
	Z_1-c_n S, & \text{otherwise}
	\end{cases}$$.		
\end{corollary}

\begin{proof}
 Note that,  $\delta_{{c_n}}(\underline{T})=Z_1-S\Psi_{c_n}(W)$, where $ \Psi_{c_n}(w)=c_n - w$. We have, for any fixed for $c_n \in \left\{(-\infty, 0) \bigcup \left[\frac{n+2}{4n^2}, \infty\right) \right\}$,
	\begin{align*}
	&\mathbb{P}_{\underline{\theta}}\left[\Bigg\{\left(\Psi_{c_n}(W)> \frac{1+nW}{4n^2}\right) \bigcap\left(W\geq \frac{1}{n}\right)\Bigg\} \bigcup \Bigg\{\left(\Psi_{c_n}(W)<-W\right) \Bigg\} \right]\\
	&=\mathbb{P}_{\underline{\theta}}\left[\left\{\frac{1}{n} \leq W < \frac{4n}{n+1}\left(c_n-\frac{1}{4n^2}\right)\right\} \bigcup \left\{c_n-W<-W\right\} \right]>0, ~\text{for some}~ \underline{\theta} \in \Theta.
	\end{align*}
  Hence, the assertion follows from Theorem 4.2.\\ 
\end{proof}

\section{Results for estimation after selection of the worst exponential population } 
 We define the population associated with the shorter guarantee time ($\theta_{1}$), the worst population. For the goal of selecting the worst population, a natural selection rule is to select the population associated with the smallest sample minimum $Z_1$. In this section, we consider the problem of estimating the location parameter of the selected exponential population (worst population). The location parameter of the selected population is
  
\begin{equation} \label{S5.E1}
\begin{split}
\mu_S &= \begin{cases}
\mu_1, & \text{if $X_1 \leq X_2$ }\\
\mu_2, & \text{if $X_2 <  X_1$}
\end{cases} 
= \mu_1I(X_1\leq X_2) + \mu_2I(X_2 < X_1)
\end{split}.
\end{equation}
In this section we consider the estimation of $\mu_{S}$ under the scaled squared error loss function
\begin{equation}
\label{S5.E2}
L_{\underline{T}}(\underline{\theta},a) = \left( \frac{a-\mu_S}{\sigma}\right)^2  ,~~ \underline{\theta} \in \Theta, ~ a \in \mathcal{A}.
\end{equation}
Note that, $\mu_{S}+\mu_{M}=\mu_{1}+\mu_{2}$ and $ \mathbb{E}_{\underline{\theta}}\left(X_i-\frac{S}{2n(n-1)}\right)=\mu_i,~ i=1,2.$ Now using Theorem 3.1., the UMVUE of $\mu_{S}$ is
\begin{align*}
\delta_{U}^*(\underline{T})&=X_1+X_2-\frac{2S}{2n(n-1)}-\delta_{U}(\underline{T})\\
&=Z_1-\frac{S}{2n(n-1)}+\frac{S}{2n(n-1)}\left(1-\frac{\Delta}{S}\right)^{2(n-1)}I\left(\frac{\Delta}{S} \leq 1\right).
\end{align*}
Let $k_2=\frac{1}{n(2n-1)}$. On using the arguments preceding Theorem 3.2 it follows that the estimator $\delta_{{k}_{2}}$ is the generalized Bayes estimator of $\mu_{S}$ under the scaled squared error loss function (5.2) and non informative prior (3.4).
Thus, we have the following theorem.
\begin{theorem}
	(a) The UMVUE of $\mu_{S}$ is given by
	$$\delta_{U}^*(\underline{T})= Z_1-\frac{S}{2n(n-1)}+\frac{S}{2n(n-1)}\left(1-\frac{\Delta}{S}\right)^{2(n-1)}I\left(\frac{\Delta}{S} \leq 1\right).$$
	(b) Under the scaled squared error loss function (5.2), the estimator $\delta_{{k}_{2}}(\underline{T})=Z_1-\frac{1}{n(2n-1)}S$ is the generalized Bayes estimator of $\mu_{S}$, with respect to the non informative prior given by (3.4).
\end{theorem}
\noindent For the goal of estimating $\mu_{S}$, any affine and permutation equivariant estimator is of the form,
\begin{equation} \label{S6,E2}
\delta_{\Psi}(\underline{T})= Z_1 - S\Psi\left(W\right),
\end{equation}
for some function $\Psi:[0, \infty)\rightarrow \mathbb{R}$. Let $\mathcal{D}_1$ denote the class of all affine and permutation equivariant estimators of the type (5.3). A natural class of estimators for estimating $\mu_{S}$ is $\mathcal{D}_{2}=\{d_{{c}_{n}}:c_{n}\in \mathbb{R}\}$, where $d_{{c}_{n}}(\underline{T})=Z_1-c_nS, c_n \in \mathbb{R}$. Let $k_1=\frac{1}{2n(n-1)}$  and $k_2=\frac{1}{n(2n-1)}$, so that $d_{{k}_{1}} \in \mathcal{D}_2$ and  $d_{{k}_{2}} \in \mathcal{D}_2$ are, respectively, the analogues of the UMVUEs and BAEEs of $\mu_{1}$ and $\mu_{2}$. Moreover, $d_{0}(\underline{T})=Z_1$ is the analogue of the MLEs of $\mu_{1}$ and $\mu_{2}$. Note that the class $\mathcal{D}_2$ is a subclass of affine and permutation equivariant estimators $\mathcal{D}_1$.  We will call the estimators in the  subclass $\mathcal{D}_{2}=\{d_{{c}_{n}}:c_{n}\in \mathbb{R}\}$ as linear, affine and permutation equivariant  estimators. 
We will now characterize admissible and inadmissible estimators in the class of estimators $\mathcal{D}_2$.
 The following lemma will be useful in this direction.
 \begin{lemma}
 	Let  $U_1=\frac{Z_1-\mu_{S}}{\sigma}$. Then
 	\begin{itemize}
 		\item [(i)] $\mathbb{E}_{\underline{\theta}}(U_1)= \frac{1}{n}\left[1-\left(\frac{\mu+1}{2}\right)e^{-\mu}\right],~\mu \geq 0$;
 		\item [(ii)] $\mathbb{E}_{\underline{\theta}}(U_1^2)=\frac{1}{n^2}\left[2-\left(\frac{\mu^2+3\mu+3}{2}\right)e^{-\mu}\right] ,~ \mu \geq 0$.
 	\end{itemize}
 \end{lemma}
 \begin{proof}
 	Follows on using Lemma 4.1, along with the following facts:
 	\begin{itemize}
 		\item[(i)] $U_1+U= \frac{X_1-\mu_{1}}{\sigma}+ \frac{X_2-\mu_{2}}{\sigma}$, where $U$ is as defined in Lemma 4.1;
 		\item[(ii)] $U_1^2+U^2= \left(\frac{X_1-\mu_{1}}{\sigma}\right)^2+ \left(\frac{X_2-\mu_{2}}{\sigma}\right)^2$;
 		\item[(iii)] $\frac{n(X_i-\mu_i)}{\sigma} \sim Exp(0,1),~i=1,2.$
 	\end{itemize}
 \end{proof}
Using Lemma 5.1., the risk function of an estimator $d_{{c_n}}(\underline{T}) \in \mathcal{D}_2$ is 
\begin{align}\label{S5.E4}
R_{\mu}\left(d_{{c}_{n}}\right)
&=\frac{1}{n^2}\left[2-\frac{(\mu^2+3\mu+3)}{2}e^{-\mu}\right] - \frac{ 4(n-1)}{n}\left[1-\left(\frac{\mu+1}{2}\right)e^{-\mu}\right]c_n \nonumber																																		\\ 
& ~~~~~~~~~~+ 2(n-1)(2n-1)c_n^2  ,~~~\mu \geq 0.
\end{align}
The risk function is clearly bowl-shaped with unique minimum at $c_n=c_n^*(\mu)$, where
\begin{equation}
c_n^*(\mu)=\frac{1}{n(2n-1)}\left[1-\left(\frac{\mu+1}{2}\right)e^{-\mu}\right],~\mu \geq 0,
\end{equation}

\begin{equation}
\inf\limits_{\mu \geq 0} c_n^*(\mu) = \frac{1}{2n(2n-1)}= k_0~ (\text{say})
\end{equation}
and
\begin{equation}
\sup\limits_{\mu \geq 0} c_n^*(\mu) = \frac{1}{n(2n-1)}= k_2 ~(\text{say}).
\end{equation}

Now we have the following theorem that characterizes admissible and inadmissible estimators in the class $\mathcal{D}_2$.

\begin{theorem}
	Let $k_0=\frac{1}{2n(2n-1)}$ and $k_2= \frac{1}{n(2n-1)}$. For the problem of estimating $\mu_{S}$, consider the
	scaled squared error loss function (5.2).
	\begin{itemize}
		\item[(a)] The estimators in the class $\mathcal{D}_{2,M}=\left\{d_{{c}_{n}}:c_n \in [k_0,k_2]\right\}$ are admissible among the estimators in the class $\mathcal{D}_2$. Moreover, the estimators in the class $\mathcal{D}_{2,1}=\left\{d_{{c}_{n}}: c_n \in (-\infty,k_0) \cup (k_2, \infty)\right\}$ are inadmissible. For any $-\infty < b_n < c_n \leq k_0$ or $k_2 \leq c_n < b_n < \infty$,
		$$R_{\mu}(d_{{c}_{n}}) < R_{\mu}(d_{{b}_{n}}),~ \forall \mu \geq 0.$$
		
		\item[(b)] The estimator $d_{{k}_{2}}(\underline{T})$ is minimax among the linear, affine and permutation equivariant estimators belonging to $\mathcal{D}_2.$	
	\end{itemize}
\begin{proof}
	$(a)$ Similar to the proof of Theorem 4.1. on using (5.5)-(5.7).\\
	$(b)$ In view of $(a)$, it is enough to find a minimax estimator among $\left\{d_{{c}_{n}}: c_n \in [k_0,k_2]\right\}$. Using (5.4) the risk function of an estimator $d_{{c_n}} \in \mathcal{D}_2$ can be written as
	\begin{align*}
	R_{\mu}\left(d_{{c}_{n}}\right)
	&=\frac{e^{-\mu}}{2n^2}\left[4n(n-1)c_n-3+ (4n(n-1)c_n-3)\mu -\mu^2\right] +\frac{2}{n^2}- \frac{ 4(n-1)}{n}c_n \nonumber																																		\\ 
	& ~~~~~~~~~~+ 2(n-1)(2n-1)c_n^2  ,~~~\mu \geq 0.
	\end{align*}
We have, for $c_n \in [k_0,k_2]= \left[\frac{1}{2n(2n-1)},\frac{1}{n(2n-1)}\right]$,
\begin{equation*}
\frac{\partial }{\partial \mu}R_{\mu}(d_{{c}_{n}})=\frac{\mu e^{-\mu}}{2n^2}\left[\mu-(4n(n-1)c_n-1)\right],~\mu \geq 0.
\end{equation*}	
Consider the following cases:\\
\textbf{Case I}: $\frac{1}{2n(2n-1)} \leq c_n \leq \frac{1}{4n(n-1)} < \frac{1}{n(2n-1)}  $	\\
In this case,
\begin{align}
 \frac{\partial }{\partial \mu} R_{\mu}(d_{{c}_{n}})& \geq 0,~ \forall \mu \geq 0 \nonumber\\
 \Rightarrow \sup\limits_{ \mu \geq 0 }  R_{\mu}(d_{{c}_{n}}) &= \lim\limits_{\mu \to \infty}  R_{\mu}(d_{{c}_{n}}) \nonumber \\
 &=2(n-1)(2n-1)c_n^2-\frac{4(n-1)}{n}c_n+\frac{2}{n^2}.
\end{align}
 \textbf{Case II}: $\frac{1}{4n(n-1)} \leq c_n   \leq \frac{1}{n(2n-1)}  $\\	
 In this case,
 \begin{align*}
 \frac{\partial }{\partial \mu} R_{\mu}(d_{{c}_{n}}) &= \frac{\mu e^{-\mu}}{2n^2}\left[\mu-(4n(n-1)c_n-1)\right]
 \end{align*}
 is negative, if $\mu \in \left[0, 4n(n-1)c_n-1\right)$ and positive, if $\mu \in \left[ 4n(n-1)c_n-1, \infty\right)$. Thus
 \begin{align}
 \sup\limits_{ \mu \geq 0 }  R_{\mu}(d_{{c}_{n}}) &= \max \left\{ R_{\mu=0}(d_{{c_n}}), \lim\limits_{\mu \to \infty}  R_{\mu}(d_{{c}_{n}}) \right\} \nonumber \\ 
 &= \max\left\{2(n-1)(2n-1)c_n^2-\frac{4(n-1)}{n}c_n+\frac{2}{n^2}+\frac{4n(n-1)c_n-3}{2n^2}, \right. \nonumber \\
 & \left.\hspace{4cm}~2(n-1)(2n-1)c_n^2-\frac{4(n-1)}{n}c_n+\frac{2}{n^2}\right\} \nonumber \\
 &= 2(n-1)(2n-1)c_n^2-\frac{4(n-1)}{n}c_n+\frac{2}{n^2}, ~~ \left(\text{as}~~ c_n \leq \frac{1}{n(2n-1)}\right)
 \end{align}
 On combining (5.8) and (5.9), we get
 $$ \sup\limits_{ \mu \geq 0 }  R_{\mu}(d_{{c}_{n}}) =2(n-1)(2n-1)c_n^2-\frac{4(n-1)}{n}c_n+\frac{2}{n^2},~ \forall ~c_n \in \left[\frac{1}{2n(2n-1)}, \frac{1}{n(2n-1)} \right],$$
which is decreasing function of $c_n$ on $\left[\frac{1}{2n(2n-1)}, \frac{1}{n(2n-1)} \right]$ with minimum at $c_n=\frac{1}{n(2n-1)}=k_2$. Hence the result follows.
\end{proof}

\end{theorem}
\begin{remark}
(a)	As a consequence of Theorem $5.2.(a)$, we conclude that the natural estimators  $d_{0}(\underline{T})=Z_1$ and $d_{k_1}=Z_1-\frac{1}{2n(n-1)}S$ are  inadmissible for estimating $\mu_{S}$ under the scaled mean squared error criterion. The estimator $d_{k_0}(\underline{T})=Z_1-\frac{1}{2n(2n-1)}S$  dominates the estimator  $d_{0}(\underline{T})$  and the estimator  $d_{k_2}(\underline{T})=Z_1-\frac{1}{n(2n-1)}S$ dominates the estimator $d_{k_1}(\underline{T})$. Moreover, the estimator $d_{k_2}(\underline{T})=Z_1-\frac{1}{n(2n-1)}S$ is admissible among linear, affine and permutation equivariant estimators belonging to the class $\mathcal{D}_2$.\\
(b) Using $(5.4)$, it follows that the natural estimator $d_{{c_n}}\in\mathcal{D}_2$ is a consistent estimator of $\mu_{S}$ if $\lim\limits_{n \to \infty}(n c_n)=0$.
\end{remark}
\noindent Now consider the class of affine and permutation equivariant estimators $\mathcal{B}_1= \left\{d_{\Psi}: \Psi: [0, \infty) \to \mathbb{R}\right\}$, where $d_{\Psi}(\underline{T})=Z_1-\Psi(W)S$.
Recall that $U_1=\frac{Z_1-\mu_{S}}{\sigma}$,  $Z=Z_2-Z_1$, $W=\frac{Z}{S}$ and $V=\frac{S}{\sigma}$. Let $f_{1,\underline{\theta}}(\cdot)$  denote the p.d.f. of $W,~ \mu \geq0$, and let $f_2(\cdot)$ denote the p.d.f. of $V \sim Gamma(2(n-1),1)$. The following lemma will be useful in establishing the result on a sufficient condition of inadmissibility for estimating $\mu_{S}$ .
\begin{lemma} 
	Let $w \in (0,\infty)$ be fixed. Then, the conditional p.d.f. of $(U_1,V)$, given $W=w$, is given by
	\begin{equation*}
	f_{4, \underline{\theta}}(u,v|w)=\begin{cases*}
	\frac{n^2}{\Gamma(2(n-1))f_{1,\underline{\theta}}(w)}v^{(2n-2)}e^{-v\left(1+nw\right)}e^{-2nu}(e^{-\mu}+e^{\mu}), & \parbox[t]{6cm}{$0<u<\frac{\mu}{n},~~\frac{\mu-nu}{nw}<v<\infty $\\ \hspace*{2cm}or,\\ $\frac{\mu}{n} < u< \infty,~~ 0<v< \infty $}\\
	\frac{n^2}{\Gamma(2(n-1))f_{1,\underline{\theta}}(w)}v^{(2n-2)}e^{-v\left(1+nw \right)}e^{-2nu}e^{-\mu}, & \parbox[t]{5.5cm}{$0<u <\frac{\mu}{n},~~ 0<v< \frac{\mu-nu}{nw}$} \\
	\end{cases*}
	\end{equation*}
	\begin{proof}
		Similar to the proof of Lemma 4.2.
	\end{proof}
\end{lemma}
The following lemmas will be useful in proving the main results of this section.
\begin{lemma}
	For any fixed positive integer $m$ and $a \in (0,\infty)$, define 
	\begin{align}
	\xi(\mu)&=\frac{1+ \mu\left[1+2\displaystyle{\int_{0}^{\mu}} e^{2t}\overline{G}_{m}\left(at \right)dt\right]-2\displaystyle{\int_{0}^{\mu}} te^{2t}\overline{G}_{m}\left(at \right)dt}{1+\displaystyle{\int_{0}^{\mu} e^{2t} \overline{G}_{m+1}\left(at \right)dt}},~ \mu \geq0,
	\end{align}
where, $\overline{G}_{m}\left(\cdot \right)$ is as defined in Lemma $4.3.$ Then,
	\begin{itemize}
		\item [(i)]   $\lim\limits_{\mu \to \infty} \displaystyle{\int_{0}^{\mu}} e^{2t}\overline{G}_{m}\left(at \right)dt = \begin{cases}
		\infty, & \text{if~ $0 < a \leq 2$ }\\
		\frac{1}{2}\left[\left(\frac{a}{a-2}\right)^{m}-1\right], & \text{if~ $a > 2 $}
		\end{cases};$
	\item [(ii)] $\lim\limits_{\mu \to \infty} \displaystyle{\int_{0}^{\mu}} t e^{2t}\overline{G}_{m}\left(at \right)dt = \begin{cases}
	\infty, & \text{if~ $0 < a \leq 2$ }\\
	\frac{a^m}{4(a-2)^m}\left[\frac{2m}{a-2}-1\right]+\frac{1}{4}, & \text{if~ $a > 2 $}
	\end{cases};$
	
	\item [(iii)] $ \lim\limits_{\mu \to \infty} \frac{\mu\left[1+2 \displaystyle{\int_{0}^{\mu}}  e^{2t}\overline{G}_{m}\left(at \right)dt \right]}{1+ \displaystyle{\int_{0}^{\mu}}  e^{2t}\overline{G}_{m+1}\left(at \right)dt }= \begin{cases}
	\frac{2m}{a}, & \text{if~ $0 < a \leq 2$ }\\
	\infty, & \text{if~ $a > 2 $}
	\end{cases};$

	\item [(iv)] $ \lim\limits_{\mu \to \infty} \frac{ \displaystyle{\int_{0}^{\mu}} t e^{2t}\overline{G}_{m}\left(at \right)dt}{1+ \displaystyle{\int_{0}^{\mu}}  e^{2t}\overline{G}_{m+1}\left(at \right)dt }= \begin{cases}
	\frac{m}{a}, & \text{if~ $0 < a \leq 2$ }\\
	\frac{\left(\frac{a}{a-2}\right)^m\left[\frac{2m}{a-2}-1\right]+1}{2\left[1+\left(\frac{a}{a-2}\right)^{m+1}\right]}, & \text{if~ $a > 2 $}
	\end{cases};$
	
	\item [(v)] $ \lim\limits_{\mu \to \infty} \xi(\mu)= \begin{cases}
	0, & \text{if~ $0 < a \leq 2$ }\\
	\infty, & \text{if~ $a > 2 $}
	\end{cases};$
	\item [(vi)] $ \inf\limits_{\mu \geq 0} \xi(\mu)= \begin{cases}
	0, & \text{if~ $0 < a \leq 2$ }\\
	1, & \text{if~ $a > 2 $}
	\end{cases};$
		\item [(vii)] $ \sup\limits_{\mu \geq 0} \xi(\mu)= \begin{cases}
	1, & \text{if ~ $0 < a \leq 2$ }\\
	\infty, & \text{if ~ $a > 2 $}
	\end{cases}$
	\end{itemize}
\end{lemma}
\begin{proof}
$(i)$ and $(ii).$ Similar to the proof of Lemma 4.3.\\
$(iii)$ Clearly, for $a>2$, using $(i)$ and $(ii)$,
$$\lim\limits_{\mu \to \infty} \frac{\mu\left[1+2 \displaystyle{\int_{0}^{\mu}}  e^{2t}\overline{G}_{m}\left(at \right)dt \right]}{1+ \displaystyle{\int_{0}^{\mu}}  e^{2t}\overline{G}_{m+1}\left(at \right)dt }= \infty $$
For $0 < a \leq 2$, we have

$\lim\limits_{\mu \to \infty} \frac{\mu\left[1+2 \displaystyle{\int_{0}^{\mu}}  e^{2t}\overline{G}_{m}\left(at \right)dt \right]}{1+ \displaystyle{\int_{0}^{\mu}}  e^{2t}\overline{G}_{m+1}\left(at \right)dt }=\lim\limits_{\mu \to \infty} \frac{\mu}{1+ \displaystyle{\int_{0}^{\mu}}  e^{2t}\overline{G}_{m+1}\left(at \right)dt}+ \lim\limits_{\mu \to \infty}\frac{2 \mu \displaystyle{\int_{0}^{\mu}}  e^{2t}\overline{G}_{m}\left(at \right)dt}{1+ \displaystyle{\int_{0}^{\mu}}  e^{2t}\overline{G}_{m+1}\left(at \right)dt}$\\
where, using L'Hôpital's rule, we get
\begin{align*}
\lim\limits_{\mu \to \infty} \frac{\mu}{1+ \displaystyle{\int_{0}^{\mu}}  e^{2t}\overline{G}_{m+1}\left(at \right)dt}&=\lim\limits_{\mu \to \infty} \frac{1}{e^{-(a-2)\mu}\sum\limits_{j=0}^{m} \frac{\left(a \mu \right)^j}{j!}}=0
\end{align*}
and 
\begin{align*}
\lim\limits_{\mu \to \infty} \frac{2 \mu \displaystyle{\int_{0}^{\mu}}  e^{2t}\overline{G}_{m}\left(at \right)dt}{1+ \displaystyle{\int_{0}^{\mu}}  e^{2t}\overline{G}_{m+1}\left(at \right)dt}&= \lim\limits_{\mu \to \infty} \frac{2\mu e^{2\mu}\overline{G}_{m}\left(a\mu\right)+2\displaystyle{\int_{0}^{\mu}}  e^{2t}\overline{G}_{m}\left(at \right)dt}{e^{2\mu}\overline{G}_{m+1}\left(a\mu \right)}\\
&=\frac{2m}{a}
\end{align*}
Hence the assertion follows.\\
$(iv)$  For $a \geq 2$ the assertion follows using $(i)$ and $(ii)$.\\
For $0 <a \leq 2$, using $(i)$, $(ii)$ and  L'Hôpital's rule, we get
\begin{align*}
\lim\limits_{\mu \to \infty} \frac{ \displaystyle{\int_{0}^{\mu}} t e^{2t}\overline{G}_{m}\left(at \right)dt}{1+ \displaystyle{\int_{0}^{\mu}}  e^{2t}\overline{G}_{m+1}\left(at \right)dt }&= \lim\limits_{\mu \to \infty} \frac{\mu e^{2\mu}\overline{G}_{m}\left(a\mu \right)}{e^{2\mu}\overline{G}_{m+1}\left(a\mu \right)}\\
&=\lim\limits_{\mu \to \infty} \frac{\mu \sum\limits_{j=0}^{m-1} \frac{e^{-a\mu}\left(a \mu \right)^j}{j!}}{\sum\limits_{j=0}^{m} \frac{e^{-a\mu}\left(a \mu \right)^j}{j!}}=\frac{m}{a}.
\end{align*}
$(v)$ Follows on using $(i)$-$(iv)$.\\
$(vi)$ and $(vii)$. For $a >2$, $\sup\limits_{\mu \geq 0} \xi(\mu)=\infty$ follows from $(v)$ and, for $0 < a \leq 2$, $\inf\limits_{\mu \geq 0} \xi(\mu)=0$ follows from $(v)$ and the fact that $\xi(\mu) \geq 0$, $ \forall~\mu \geq0$.\\
Note that $\xi(0)=1$. Now to show that, for $a >2$ ($0 < a \leq 2)$, $\inf\limits_{\mu \geq 0} \xi(\mu)=1$  ($\sup\limits_{\mu \geq 0} \xi(\mu)=1$), it suffices to show that, for $a > 2 ~(0 < a \leq 2)$
$$\xi(\mu) \geq (\leq)~ 1,~ \forall \mu \geq 0,$$
\begin{align*}
\text{or},~ \xi_1(\mu) &= \mu +2\displaystyle{\int_{0}^{\mu} \overline{G}_{2n}\left(at \right)(\mu-t)e^{2t}dt}-\displaystyle{\int_{0}^{\mu}\overline{G}_{2n+1}\left(at \right)e^{2t}dt}~ \geq (\leq)~ 0,~ \forall \mu \geq 0.
\end{align*}
Note that,
\begin{align*}
& \xi_1(0)=0;\\
& \xi_1'(\mu)=1+2\displaystyle{\int_{0}^{\mu} \overline{G}_{2n}\left(at \right)e^{2t}dt}- \overline{G}_{2n+1}\left(a\mu \right)e^{2\mu},~ \mu \geq 0\\
&\xi_1'(0)=0~ \text{and}\\
&\xi_1''(\mu)=(a-2)\frac{e^{-a\mu} (a\mu)^{2n}}{\Gamma(2n+1)} \geq (\leq)~0,~\forall \mu \geq 0, ~\text{provided}~ a >2~ (0 < a \leq 2).
\end{align*}
Thus, for $a >2 ~(0 < a \leq 2)$,  $\xi_1'(\mu) \geq ( \leq)~ \xi_1'(0)=0$ and $\xi_1(\mu) \geq ( \leq) ~\xi_1(0)=0$.
Hence the result follows.
\end{proof}
\noindent For any fixed $w \in (0, \infty)$ and $\mu \geq 0$, the conditional risk (given $W=w$) of any affine and permutation equivariant estimator  $\delta_{\Psi}(\underline{T})=Z_1-\Psi(W)S$, given by
$$R_1\left(\mu, \Psi(w)\right)=\mathbb{E}_{\underline{\theta}}\left[\left(\frac{Z_1-S\Psi(w)-\mu_{S}}{\sigma}\right)^2 \bigg|W=w\right],$$
is minimized for the following choice of $\Psi(\cdot)$,
\begin{align*}
\Psi_{\mu}(w)&= \frac{\mathbb{E}_{\underline{\theta}}\left(\left(\frac{Z_1-\mu_{S}}{\sigma}\right)\frac{S}{\sigma}\bigg|W=w\right)}{\mathbb{E}_{\underline{\theta}}\left(\frac{S^2}{\sigma^2}\bigg|W=w\right)}\\
&=\frac{\mathbb{E}_{\underline{\theta}}\left(U_1 V\bigg|W=w\right)}{\mathbb{E}_{\underline{\theta}}\left(V^2\bigg|W=w\right)}.  \nonumber
\end{align*}
After some tedious algebra, using Lemma 5.2., we get
$$ \Psi_{\mu}(w)=\frac{1+nw}{4n^2} k_w(\mu),~ \mu \geq 0,$$
where $k_w(\mu)$ is the same as $\xi(\mu)$, with $a=\frac{1+nw}{nw}$, defined in (5.10).\par 
\noindent Using $(vi)$ and $(vii)$ of Lemma 5.3., we have
\begin{equation}
\inf\limits_{\mu \geq 0} \Psi_{\mu}(w)= \begin{cases}
0, & \text{if $w \geq \frac{1}{n}$ }\\
\frac{1+nw}{4n^2}, & \text{if $0 < w < \frac{1}{n} $}
\end{cases}=\Psi_{*}(w),~ (\text{say})
\end{equation}
and 
\begin{equation}
\sup\limits_{\mu \geq 0} \Psi_{\mu}(w)= \begin{cases}
\frac{1+nw}{4n^2}, & \text{if $w \geq \frac{1}{n}$ }\\
\infty, & \text{if $0 < w < \frac{1}{n} $}
\end{cases}=\Psi^{*}(w),~ (\text{say}).
\end{equation}
The theorem below is an analogue of Theorem 4.2, and provides a sufficient condition
for  inadmissibility of an arbitrary affine and permutation equivariant estimators of $\mu_{S}$. The proof of the theorem is omitted as it is similar to that of Theorem 4.2.
\begin{theorem}
	Suppose that $\delta_{\Psi}\left(\underline{T}\right)=Z_1-S\Psi\left(W\right) $ is any estimator of $\mu_{S}$ in the class $\mathcal{B}_1$ of affine and permutation equivariant estimators, where $W=\frac{Z_2-Z_1}{S}$. Let  $\Psi:[0,\infty) \to \mathbb{R}$ be such that
	$$\mathbb{P}_{\underline{\theta}}\left[\Bigg\{\left(\Psi_*(W)>\Psi(W)\right)\Bigg\}\bigcup\Bigg\{\left(\Psi^*(W)<\Psi(W)\right)\Bigg\}\right] >0, ~\text{for some}~~ \underline{\theta} \in \Theta ,$$
	where $\Psi_*(w)$ and $\Psi^*(w)$ are defined by $(5.11)$ and $(5.12)$ respectively. Then the estimator $\delta_{\Psi}(\cdot)$ is inadmissible for estimating $\mu_{S}$ and is dominated by
	  \begin{equation}
	\delta_{\Psi}^I\left(\underline{T}\right)=\begin{cases}
	Z_1-S\Psi_*(W), & \text{ if } \Psi(W) < \Psi_*(W) ,\\
	Z_1-S\Psi^*(W), & \text{ if } \Psi(W) > \Psi^*(W) ,\\
	\delta_{\Psi}\left(\underline{T}\right), & \text{ otherwise}.
	\end{cases}
	\end{equation}
\end{theorem}
\begin{remark}
As a consequence of Theorem $5.3.$, it follows that all the natural estimator of $\mu_{S}$ belonging to the class $\mathcal{D}_2$ are inadmissible for estimating $\mu_{S}$. The  estimator $\delta_{c_n}^I(\cdot)$ dominates the estimator $\delta_{c_n}$, where for $c_n < 0$
	$$  ~~~\delta_{c_n}^I(\underline{T})= \begin{cases}
	Z_1, & \text{ if } W \geq \frac{1}{n} ,\\
	Z_1-\frac{1+nW}{4n^2}S, & \text{ if } 0 < W < \frac{1}{n}
	\end{cases} $$
for $0 \leq c_n < \frac{1}{2n^2}$,
$$  ~~~\delta_{c_n}^I(\underline{T})= \begin{cases}
Z_1-\frac{1+nW}{4n^2}S, & \text{ if } \max\left\{0,\frac{4n^2c_n-1}{n}\right\} < W < \frac{1}{n} ,\\
Z_1-c_n S, & \text{ otherwise }
\end{cases} $$
and for $ c_n \geq \frac{1}{2n^2}$,
$$  ~~~\delta_{c_n}^I(\underline{T})= \begin{cases}
Z_1-\frac{1+nW}{4n^2}S, & \text{ if } \frac{1}{n} < W < \frac{4n^2c_n-1}{n}  ,\\
Z_1-c_n S, & \text{otherwise}.
\end{cases} $$
\end{remark}
\vspace{4mm}

\section{Simulation Results}
In this section, we perform a simulation study using $R$ software to observe the performance of some of the proposed estimators of $\mu_{M}$ and $\mu_{S}$, in terms of scaled mean squared error (mse). For convenience in presentation, we rename the estimators of $\mu_{M}$, $\delta_{0}(\underline{T})=Z_2$, $\delta_{k_1}(\underline{T})=Z_2-\frac{S}{2n(n-1)}$ and   $\delta_{k_2}(\underline{T})=Z_2-\frac{S}{n(2n-1)}$ as $\delta_{0}$, $\delta_{1}$ and $\delta_{2}$ respectively. The comparisons of these estimators are made for distinct combinations of $\mu=\frac{\theta_{2}-\theta_{1}}{\sigma}$ and $n$. For the purpose of computing mse of the estimators, we generate twenty thousand random samples of size $n=3,5,10,15$ each from two exponential populations having different location parameters and a common scale parameter. Based on the complete sufficient statistic, the mean squared error with respect to the scaled squared error loss function have been computed and compared.  
The simulated mean squared error values of the estimators of $\mu_{M}$ with respect to the loss \eqref{S2.E2}, are plotted in Figure 6.1-6.4. For estimating $\mu_{S}$, the estimators $d_{0}(\underline{T})=Z_1$, $d_{k_1}(\underline{T})=Z_1-\frac{S}{2n(n-1)}$ and $d_{k_2}(\underline{T})=Z_2-\frac{S}{n(2n-1)}$ are renamed as $\delta_{3}$, $\delta_{4}$ and $\delta_{5}$ respectively. We denote $\delta_{3}^I$, $\delta_{4}^I$ and $d_{5}^I$ as the improved estimator of $\delta_{3}$, $\delta_{4}$ and $\delta_{5}$ respectively. The simulated mean squared error values of the estimators of $\mu_{S}$ with respect to the loss \eqref{S5.E2}, are plotted in Figure 6.5-6.12.
We draw the following conclusions from the simulation study:
\begin{itemize}

	\item[(i)] For estimating $\mu_{M}$, the estimator $\delta_{0}$ (which is the natural analogue of the MLEs of $\mu_{1}$ and $\mu_{2}$) is uniformly dominated by $\delta_{1}$ and $\delta_{2}$ (the analogues of UMVUEs and BAEEs of $\mu_{1}$ and $\mu_{2}$) for different configurations of sample size i.e., in terms of  mean squared error, it performs worst among all the natural estimators under consideration.
	
	\item[(ii)] Under the criterion of mean squared error, for smaller values of $\mu$ and $n$ the estimator  $\delta_{1}$ (the analogue of UMVUEs  of $\mu_{1}$ and $\mu_{2}$) performs better than $\delta_{2}$ (the analogue of BAEEs  of $\mu_{1}$ and $\mu_{2}$). Otherwise, the estimators  $\delta_{1}$ and $\delta_{2}$ have similar mean squared error performance for estimating $\mu_{M}$. As the sample size increases, the mse of the estimators $\delta_{1}$ and $\delta_{2}$ becomes very much close. 
	
	\item[(iii)] For estimating $\mu_{S}$, the region of dominance of the estimator $\delta_{3}^I$ over $\delta_{3}$ increases with the sample size. 
	
	 \item[(iv)] For estimating $\mu_{S}$, the estimators $\delta_{4}^I$ over $\delta_{5}^I$ (which are the improvements of $\delta_{4}$ over $\delta_{5}$) yield only marginal gains in terms of mean squared error over $\delta_{4}$ and $\delta_{5}$ respectively.

	 \item[(v)] For estimating $\mu_{S}$, the estimators $\delta_{4}$, $\delta_{5}$, $\delta_{4}^I$ and $\delta_{5}^I$ have better mean squared error performance than $\delta_{3}$ and $\delta_{3}^I$.
	\item [(vi)]  As the sample size increases, the mean squared error  values of all the estimators of $\mu_{M}$ and $\mu_{S}$ under consideration approaches to zero, i.e., all these estimators are consistent.
\end{itemize}

\FloatBarrier
\begin{figure}[!h]
	\centering
	\includegraphics[width=6in]{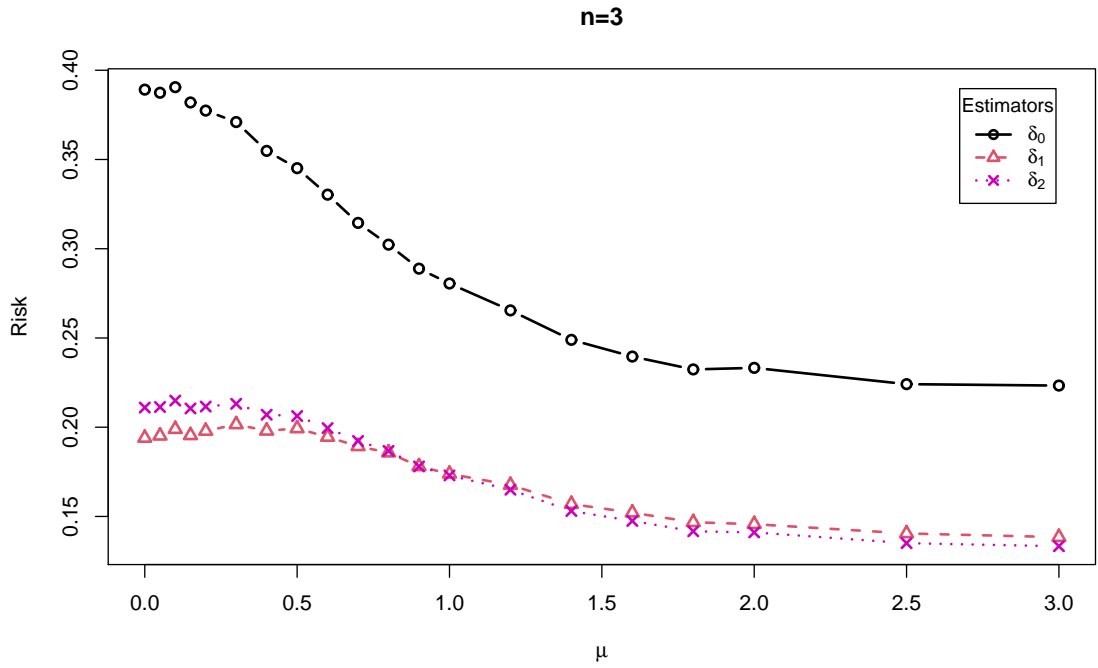}
	\caption{\textbf{Risk plots of estimators $\delta_{0}$, $\delta_{1}$ and $\delta_{2}$ for estimating $\mu_{M}$, n=3 }}
\end{figure}

\begin{figure}[!h]
	\centering
	\includegraphics[width=6in]{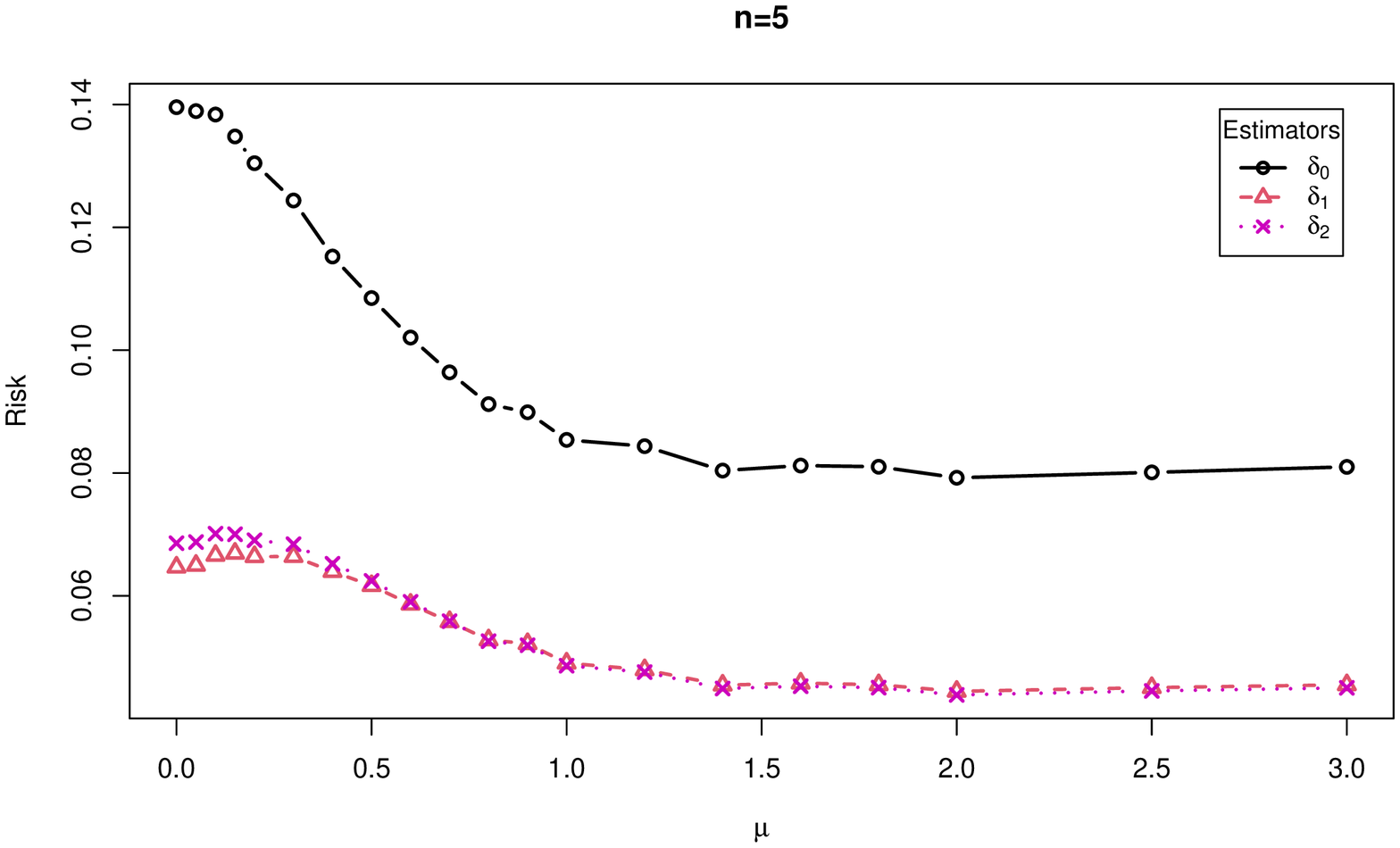}
	\caption{\textbf{Risk plots of estimators $\delta_{0}$, $\delta_{1}$ and $\delta_{2}$ for estimating $\mu_{M}$, n=5}}
\end{figure}
\begin{figure}[!h]
	\centering
	\includegraphics[width=6in]{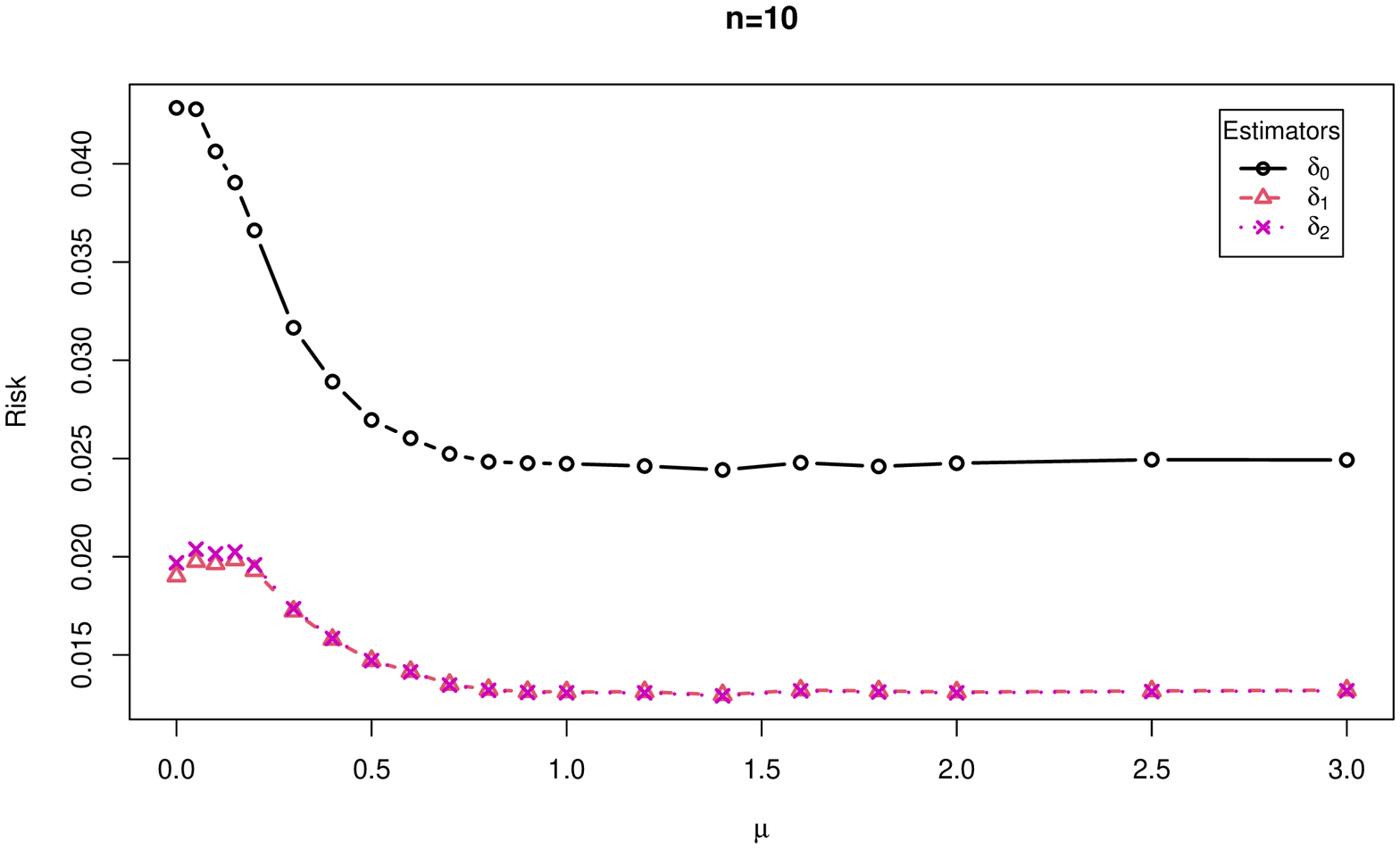}
	\caption{\textbf{Risk plots of  estimators $\delta_{0}$, $\delta_{1}$ and $\delta_{2}$ for estimating $\mu_{M}$, n=10}}
\end{figure}

\begin{figure}[!h]
	\centering
	\includegraphics[width=6in]{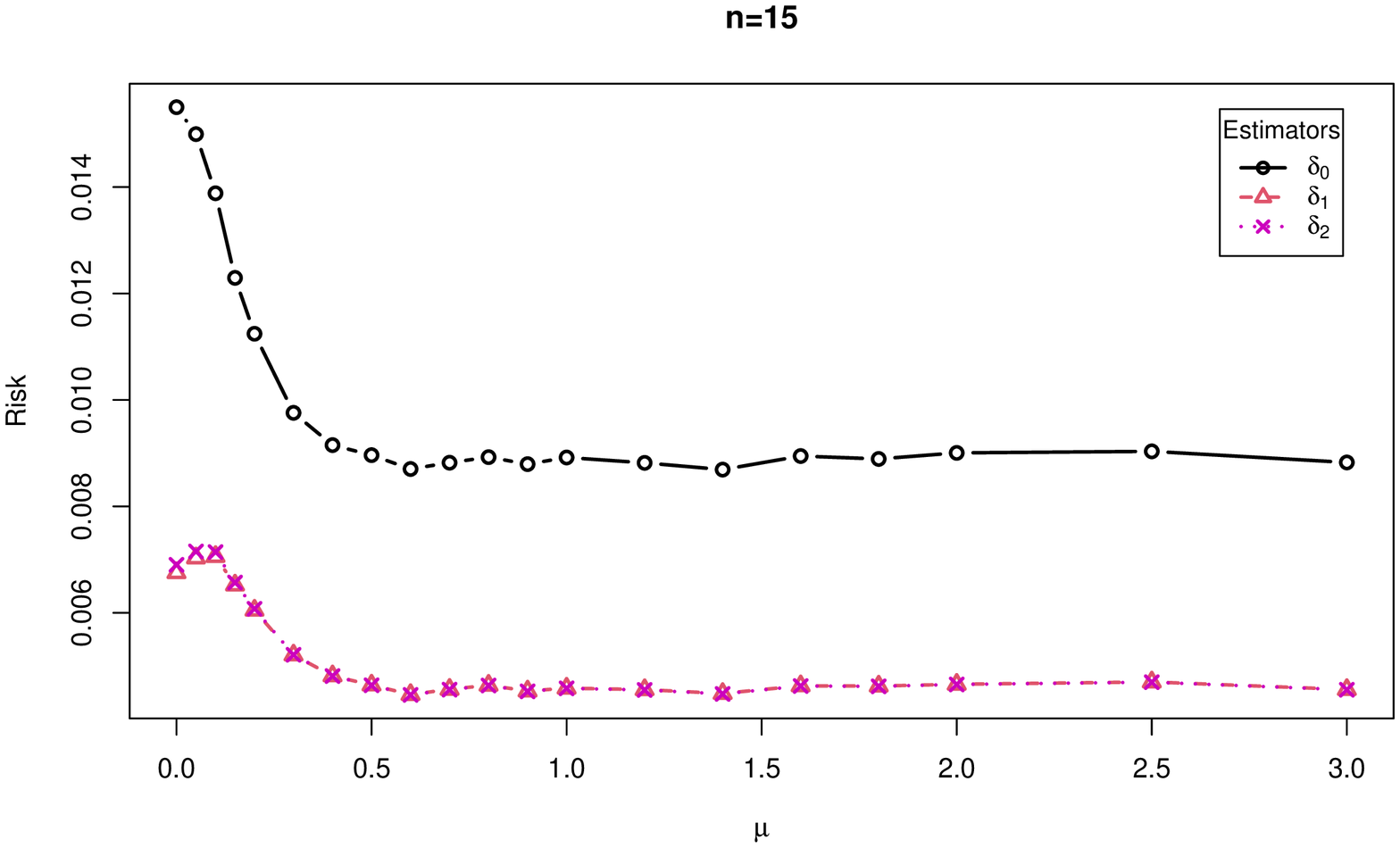}
	\caption{\textbf{Risk plots of  estimators $\delta_{0}$, $\delta_{1}$ and $\delta_{2}$ for estimating $\mu_{M}$, n=15}}
\end{figure}


\begin{figure}[!h]
	\centering
	\includegraphics[width=6in]{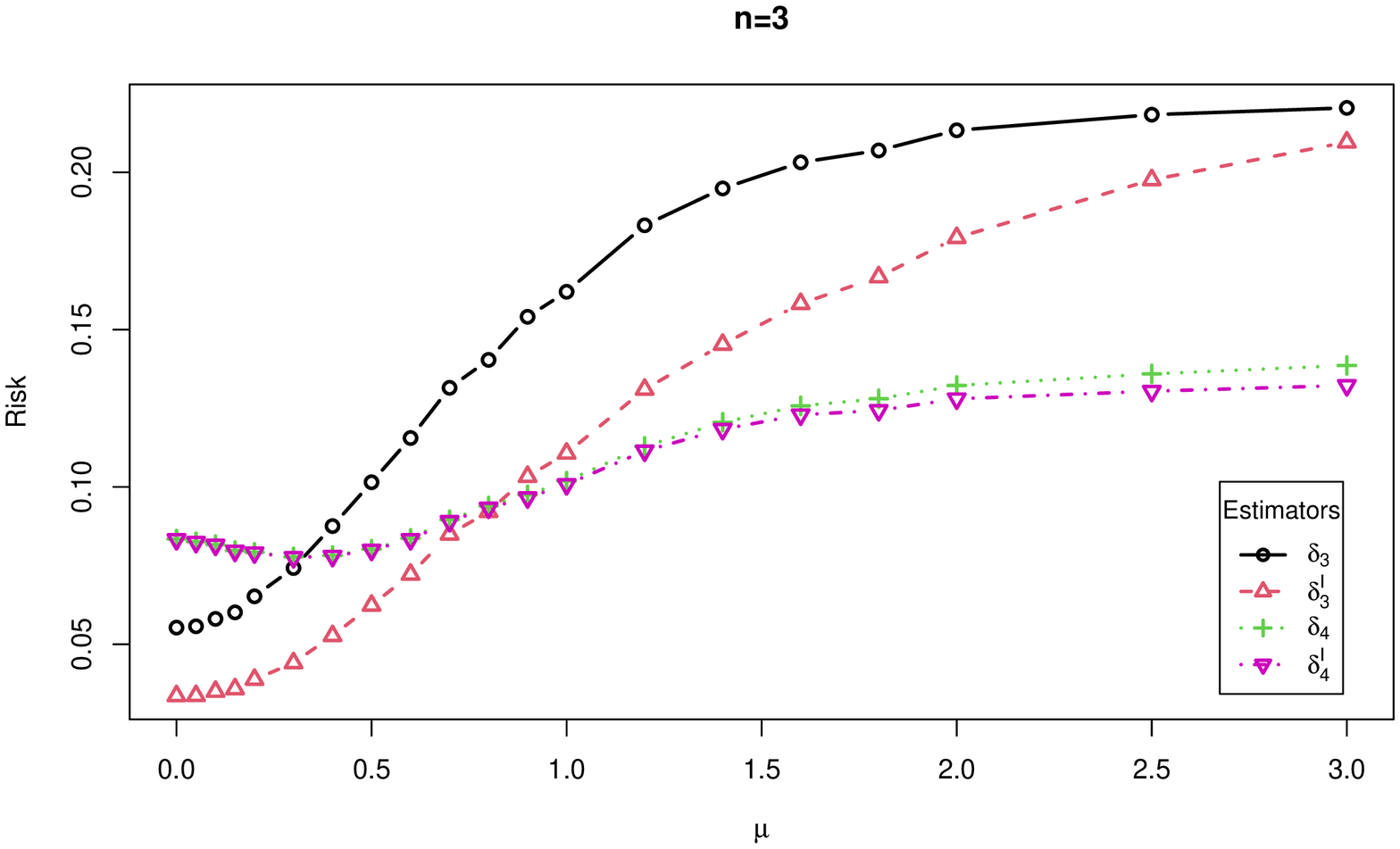}
	\caption{\textbf{Risk plots of estimators $\delta_{3}$, $\delta_{3}^I$, $\delta_{4}$ and $\delta_{4}^I$ for estimating $\mu_{S}$, n=3}}
\end{figure}

\begin{figure}[!h]
	\centering
	\includegraphics[width=6in]{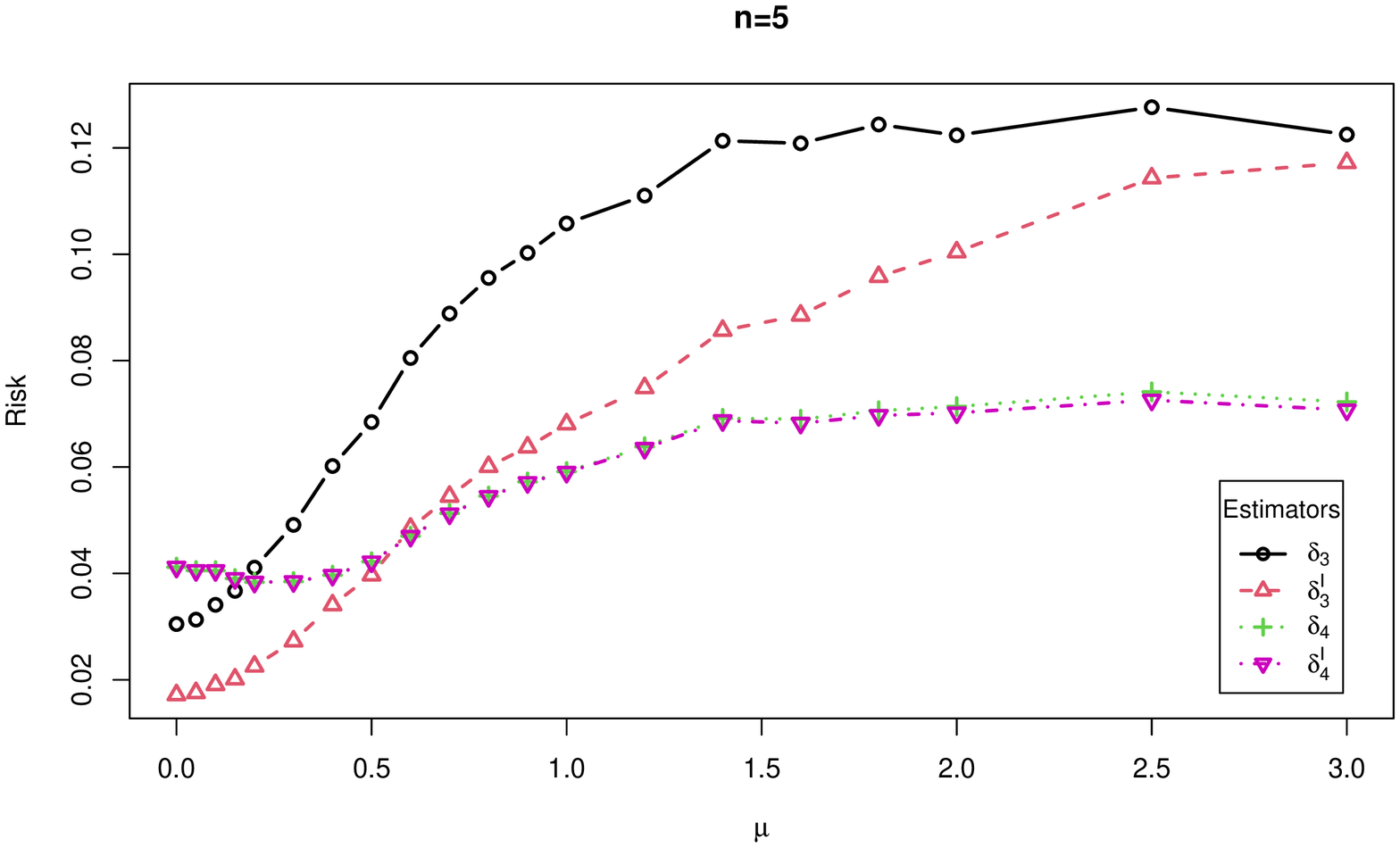}
	\caption{\textbf{Risk plots of estimators $\delta_{3}$, $\delta_{3}^I$, $\delta_{4}$ and $\delta_{4}^I$ for estimating $\mu_{S}$,  n=5}}
\end{figure}

\begin{figure}[!h]
	\centering
	\includegraphics[width=6in]{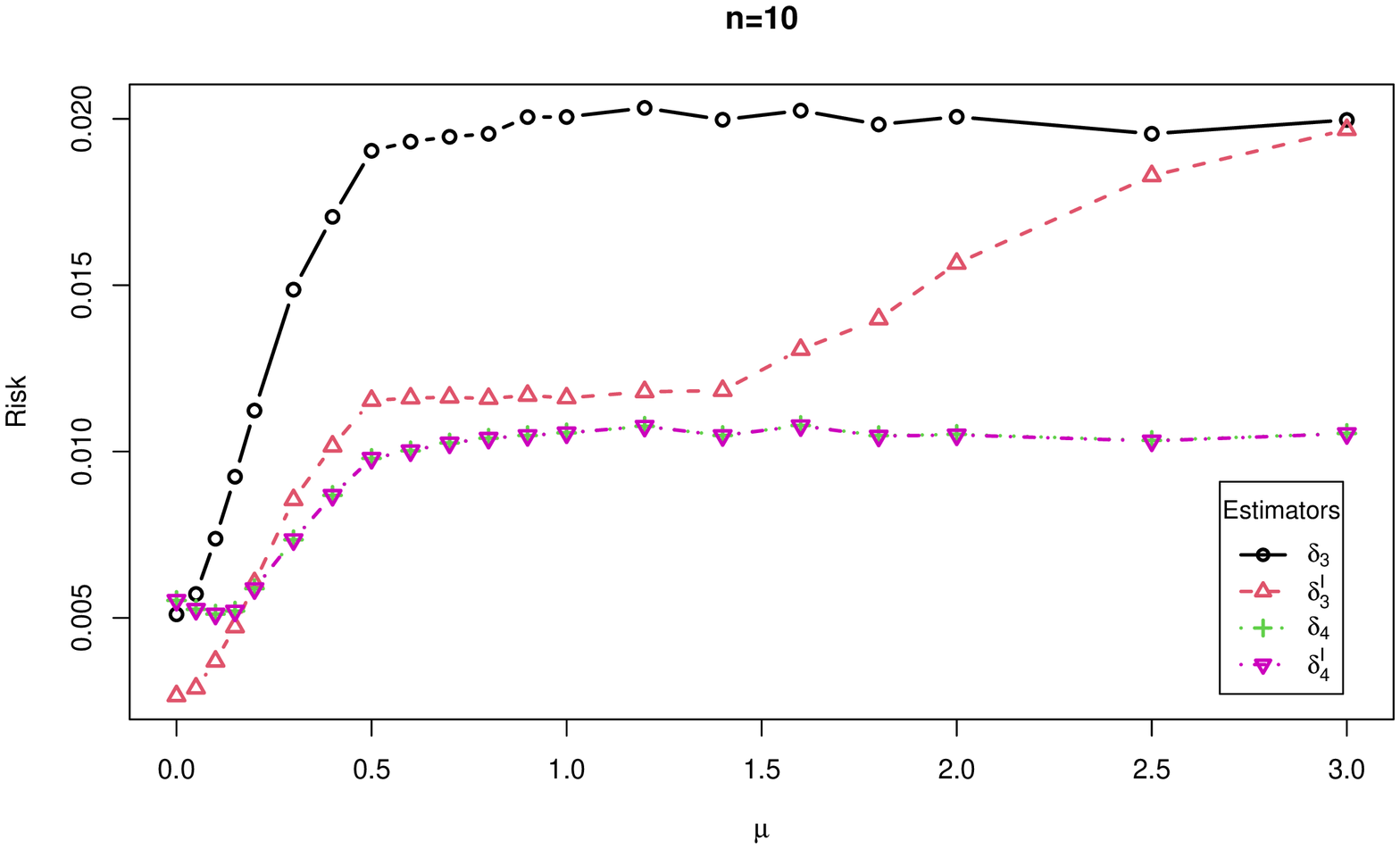}
	\caption{\textbf{Risk plots of  estimators $\delta_{3}$, $\delta_{3}^I$, $\delta_{4}$ and $\delta_{4}^I$ for estimating $\mu_{S}$ , n=10}}
\end{figure}

\begin{figure}[!h]
	\centering
	\includegraphics[width=6in]{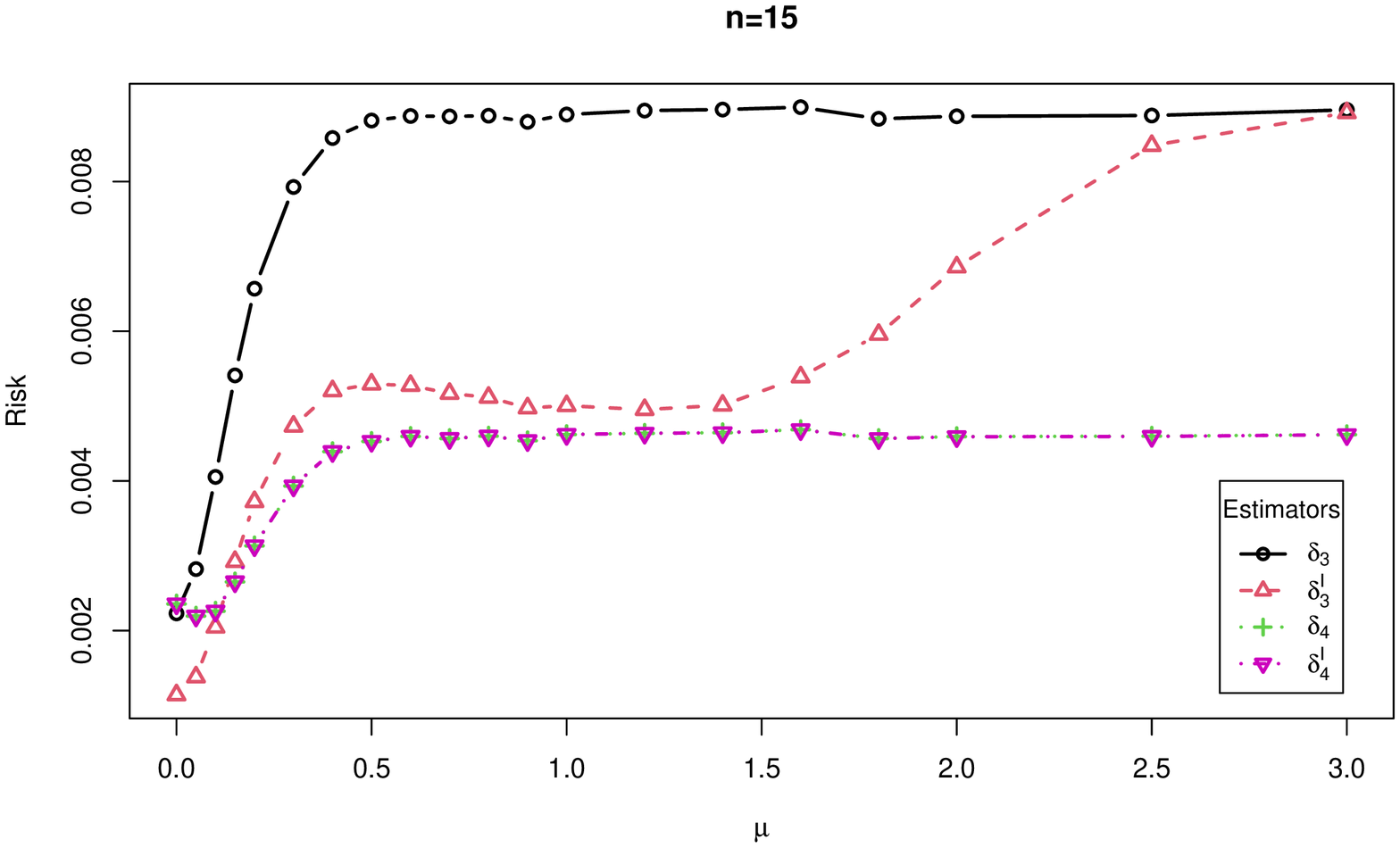}
	\caption{\textbf{Risk plots of  estimators $\delta_{3}$, $\delta_{3}^I$, $\delta_{4}$ and $\delta_{4}^I$ for estimating $\mu_{S}$ ,n=15}}
\end{figure}

\begin{figure}[!h]
	\centering
	\includegraphics[width=6in]{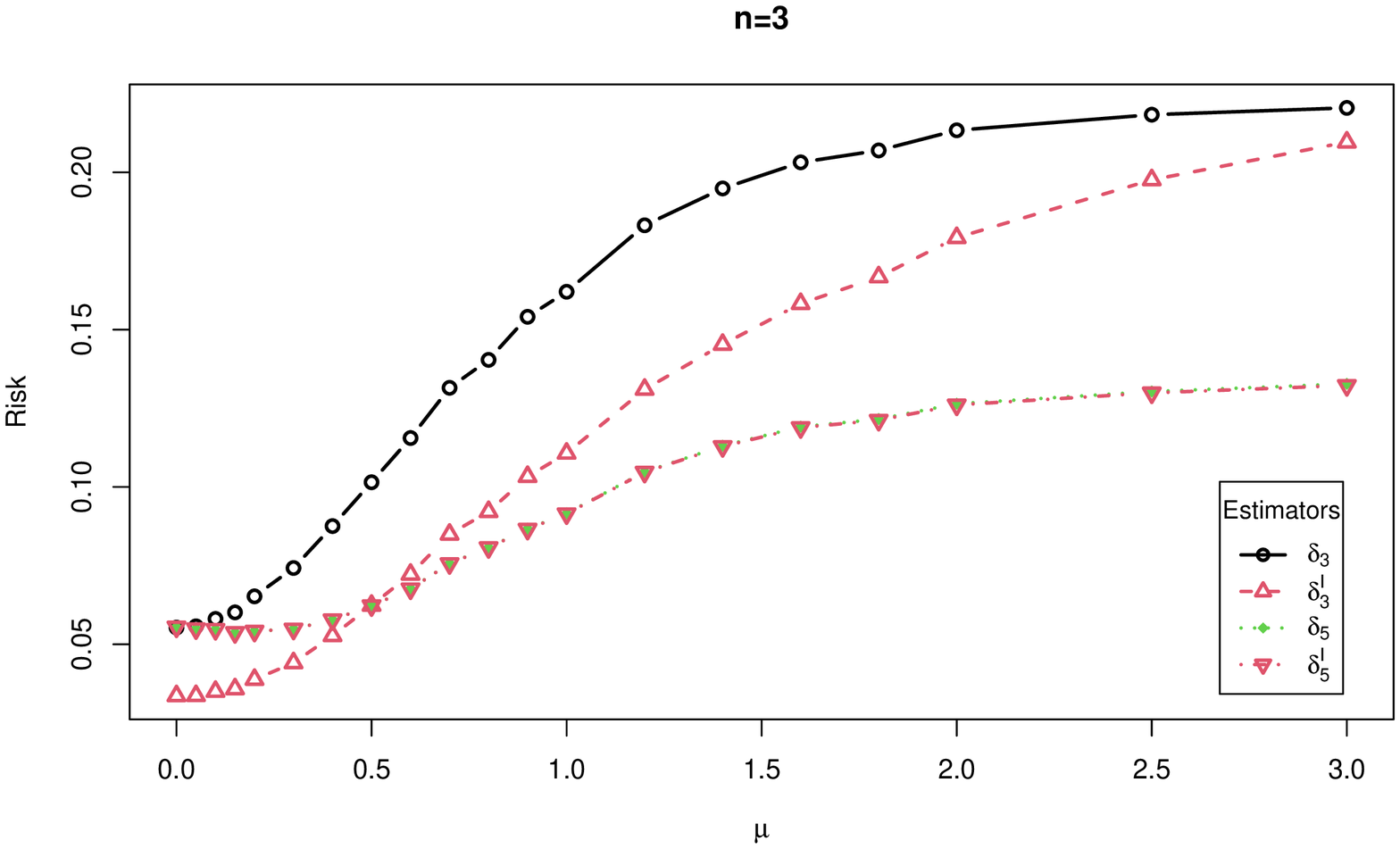}
	\caption{\textbf{Risk plots of estimators $\delta_{3}$, $\delta_{3}^I$, $\delta_{5}$ and $\delta_{5}^I$  for estimating $\mu_{S}$, n=3}}
\end{figure}

\begin{figure}[!h]
	\centering
	\includegraphics[width=6in]{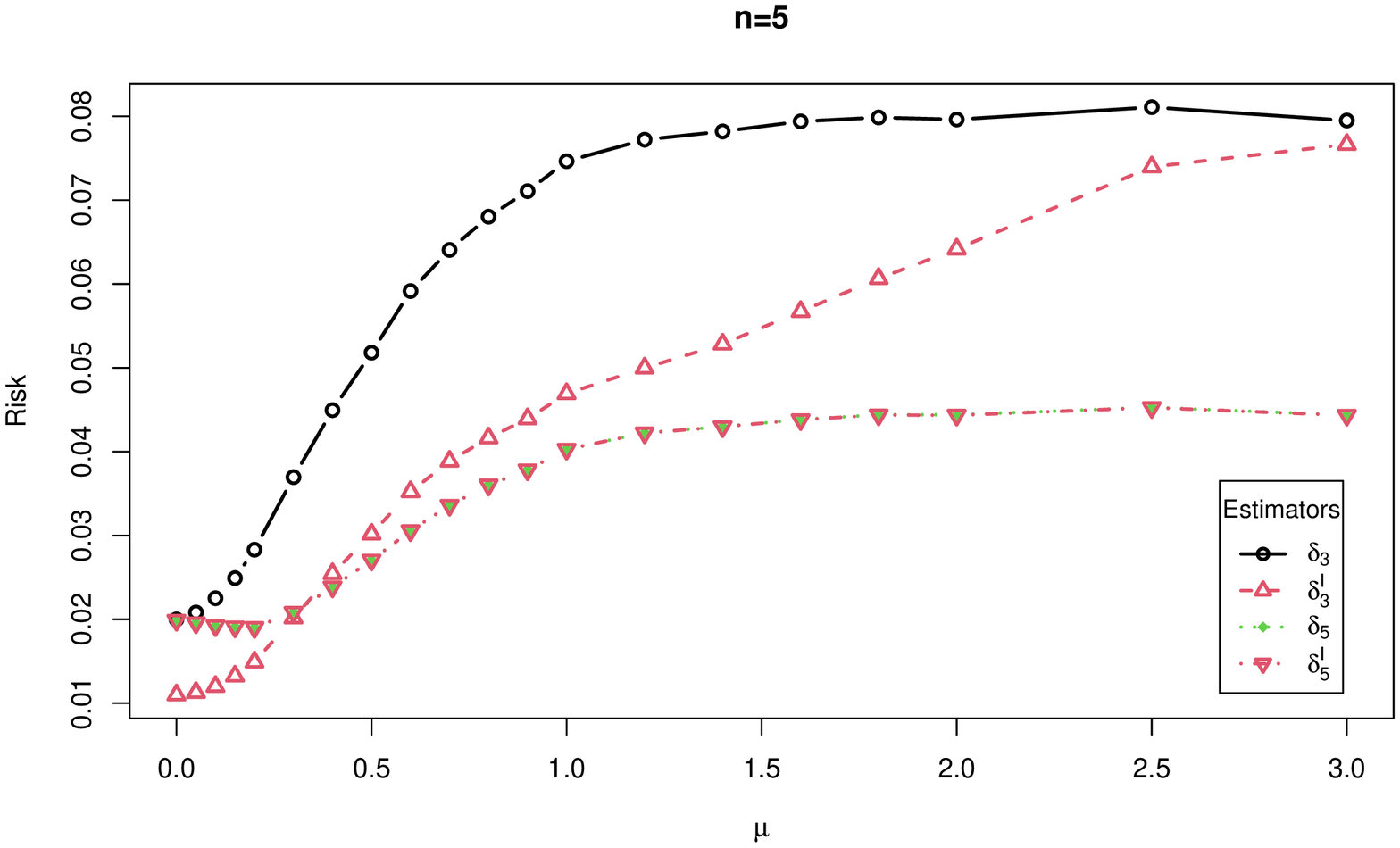}
	\caption{\textbf{Risk plots of estimators $\delta_{3}$, $\delta_{3}^I$, $\delta_{5}$ and $\delta_{5}^I$  for estimating $\mu_{S}$, n=5}}
\end{figure}
\begin{figure}[!h]
	\centering
	\includegraphics[width=6in]{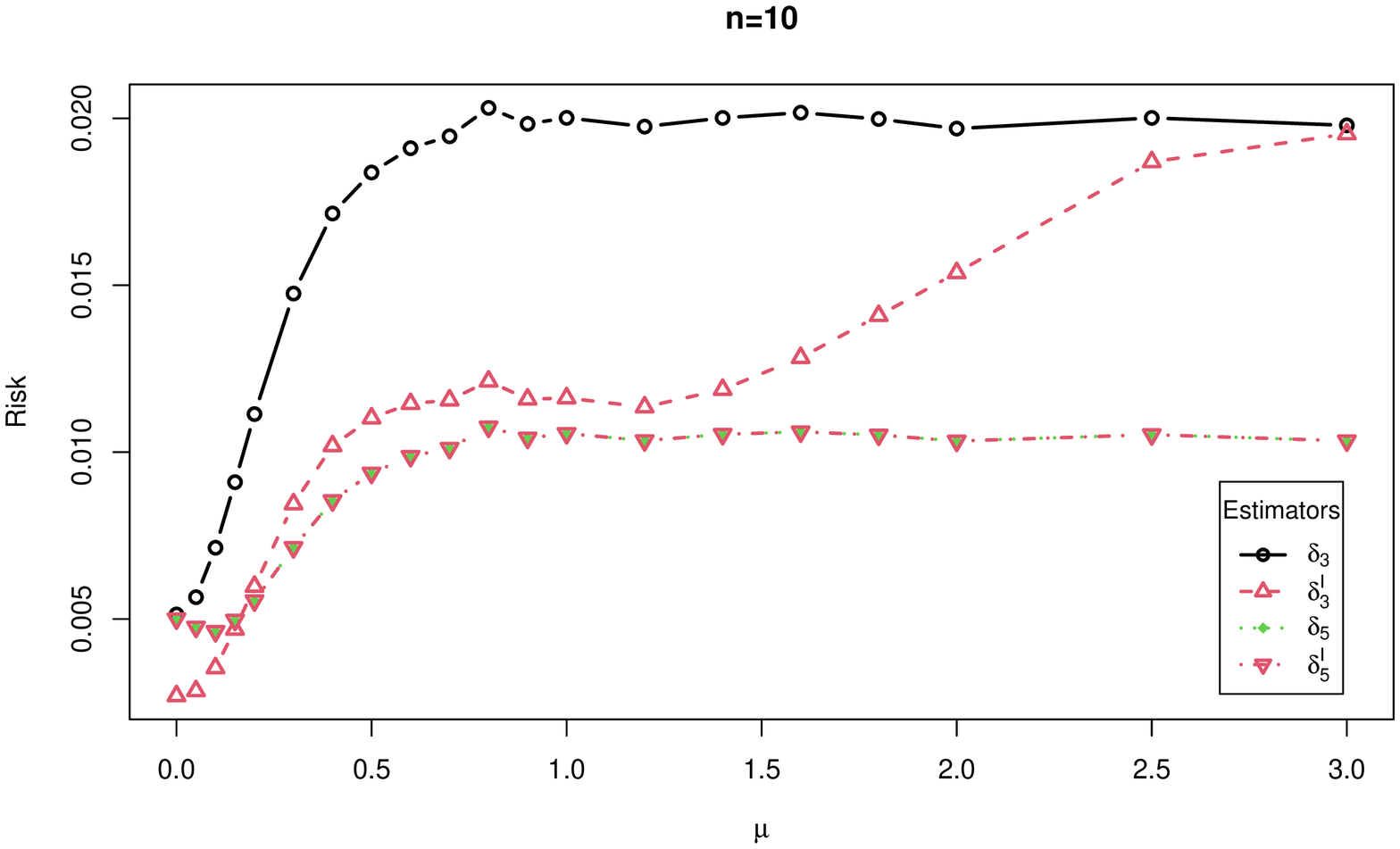}
 \caption{\textbf{Risk plots of  estimators $\delta_{3}$, $\delta_{3}^I$, $\delta_{5}$ and $\delta_{5}^I$ for estimating $\mu_{S}$ , n=10 }}
\end{figure}
\begin{figure}[!h]
	\centering
	\includegraphics[width=6in]{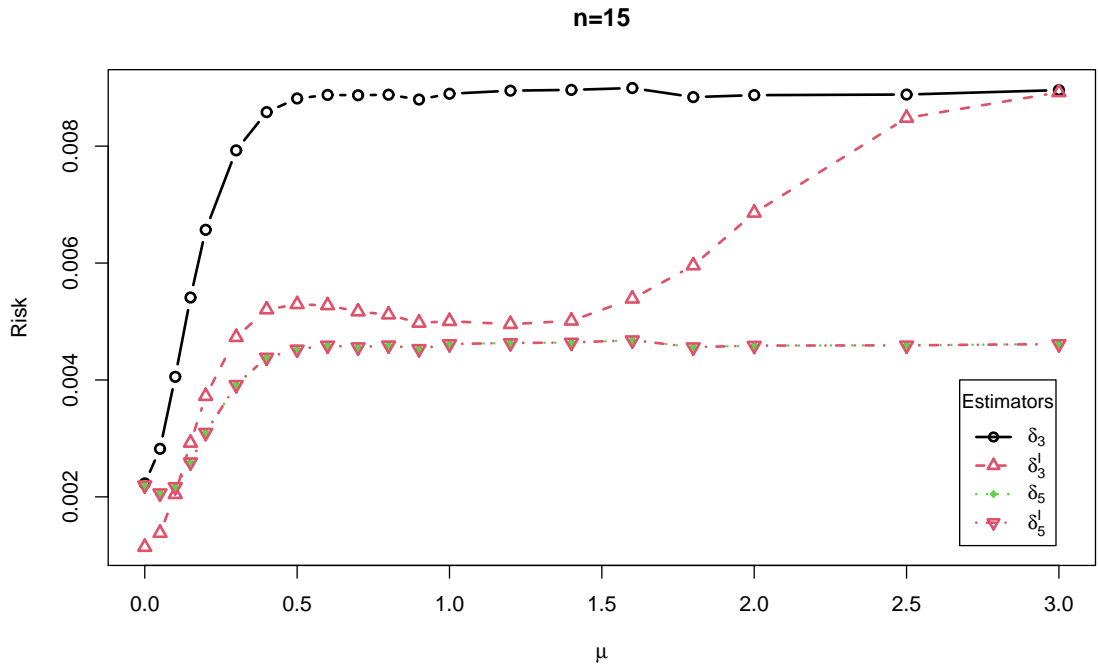}
	\caption{\textbf{Risk plots of  estimators $\delta_{3}$, $\delta_{3}^I$, $\delta_{5}$ and $\delta_{5}^I$ for estimating $\mu_{S}$ , n=15}}
\end{figure}
\FloatBarrier
\section{Closing remarks}
Consider two exponential populations with unknown guarantee times and a common unknown failure rate. Independent random samples of equal size are taken from these two populations. For the purpose of selecting the population with the longer (shorter) guarantee time, we consider a natural selection rule, which selects the population yielding the larger (smaller) sample minimum. It follows from \cite{bahadur1952impartial}, \cite{eaton1967some} and \cite{misra1994non} that this natural decision rule has various optimality properties. We have studied the problem of estimating the guarantee time of the selected population, under the scaled mean squared error criterion and obtained various decision theoretic results. We have obtained the UMVUE of the guarantee time of selected exponential population. We also characterize admissible/inadmissible estimators in the class of linear, affine and permutation equivariant estimators and find restricted minimax estimators in this class. Sufficient conditions for inadmissibility of any affine and permutation equivariant estimators are derived and dominating estimators are obtained. Finally, a simulation study is carried out to compare the performances of various competing estimators.\par  Under the same set up as ours, \cite{vellaisamy2003quantile} considered the estimators of the form $\delta_{\Psi}(\underline{T})=S\Psi\left(\frac{Z_2}{S}\right)$ for estimating $\mu_{M}$, for some function $\Psi: \mathbb{R} \to \mathbb{R}$. He derived a sufficient condition, based on the method of differential inequalities, for the inadmissibility of an estimator of the type $\delta_{\Psi}$ and obtained some dominating estimator which are not easily expressible in closed form. It is also not obvious whether one can really obtain dominating estimators that are affine and permutation equivariant. Whereas, we have considered the class of affine and permutation equivariant estimators of $\mu_{M}$ and derived a sufficient condition for the inadmissibility of an arbitrary affine and permutation equivariant estimator. As a consequence of this result various natural estimators are shown to be inadmissible and dominating estimators (having closed form expressions) are obtained. \par 
 We have not been able to obtain a global minimax estimator for estimating the guarantee time of the selected population. It would be interesting to find a minimax estimator for this problem.  Another important extension would be to extend the results (especially those obtained in Section 4 and Section 5) to $k~(\geq 2)$ populations.
These problems seem to be difficult ones and further research is needed in these directions. With obvious modifications, several results obtained in this paper can be extended to the problem of estimating quantile of the selected population.

\bibliographystyle{apalike}
\bibliography{biblography}

\end{document}